\documentclass[11pt]{amsart}

\usepackage{amsmath, amsthm, amssymb, mathrsfs}
\usepackage[hidelinks]{hyperref}
\usepackage[parfill]{parskip}    \parskip = 0.2cm    
\usepackage[margin=1in]{geometry} 
\usepackage{tikz}
\usepackage{tikz-cd}
\usetikzlibrary{matrix}
\usepackage{comment} 
\usepackage[activate={true,nocompatibility},final,tracking=true,kerning=true,spacing=true,factor=1100,stretch=10,shrink=10]{microtype}
\usepackage{bbm}
\usepackage{dsfont} 
\usepackage{siunitx} 

\usepackage{commutative-diagrams}

\usepackage{mathtools}

\graphicspath{ {images/} }
\usepackage{import}

\numberwithin{equation}{section}

\title[Unique ergodicity, and not, for primitive substitutions with isolated points]{\bf Unique ergodicity, and not, for primitive substitutions on compact alphabets containing isolated points}
\date{\today}

\author{James J.\ Walton} 
\address{School of Mathematical Sciences, Mathematical Sciences Building, University Park, Nottingham, NG7 2RD, United Kingdom}
\email{Jamie.Walton@nottingham.ac.uk}
\urladdr{https://www.nottingham.ac.uk/mathematics/people/jamie.walton}

\theoremstyle{plain}
\newtheorem{theorem}{Theorem}[section]
\newtheorem{proposition}[theorem]{Proposition}
\newtheorem{lemma}[theorem]{Lemma}
\newtheorem{corollary}[theorem]{Corollary}

\theoremstyle{definition}
\newtheorem{definition}[theorem]{Definition}
\newtheorem{remark}[theorem]{Remark}
\AtEndEnvironment{remark}{\null\hfill\exend}

\newtheorem{example}[theorem]{Example}
\AtEndEnvironment{example}{\null\hfill\exend}

\newcommand{\Z}{{\mathbb Z}}

\newcommand{\R}{{\mathbb R}}
\newcommand{\N}{{\mathbb N}}

\newcommand{\mc}{\mathcal}
\newcommand{\A}{\mc A}

\newcommand{\freq}{\operatorname{freq}}

\newcommand{\sub}{\varrho}

\newcommand{\bbo}{\mathds{1}} 

\newcommand{\avg}{\mathrm{avg}}
\newcommand{\lang}{\mathcal{L}}
\newcommand{\ths}[1]{ \num[group-separator={,}]{#1} } 

\newcommand{\exend}{\hfill \ensuremath{\Diamond}}

\begin{document}

\keywords{substitutions, infinite alphabets, positive operators, unique ergodicity}
\subjclass[2020]{37B10, 37A25, 47B65, 52C23
}
\begin{abstract}
We consider generalised subshifts generated by continuous substitution on compact Hausdorff alphabets. Although primitivity still implies minimality of the subshift generated by the substitution, Durand, Ormes and Petite showed, in contrast to the finite case, that primitivity no longer implies unique ergodicity, by constructing counterexamples with Cantor alphabet. Here we show that, even for the arguably simplest case of the one-point compactification of the natural numbers, primitivity is still insufficient for unique ergodicity, or even the existence of a natural length function (NLF). In previous work with Ma\~{n}ibo and Rust we showed that, for irreducible substitutions, unique ergodicity and existence of a NLF follow from strong power convergence of the renormalised substitution operator \(T\); criteria were also developed that can sometimes confirm strong power convergence. Here we show the converse, so that strong power convergence of \(T\) is equivalent to having both unique ergodicity and a NLF. We then consider the case of alphabets with only finitely many accumulation points, showing how upper bounds (and usually an exact formula) for the essential spectral radius of \(T\) can be derived from associated finite substitutions, determined by the behaviour of substitution of the accumulation points. For many primitive substitutions this can be used to prove quasi-compactness of \(T\), and thus unique ergodicity. For primitive substitutions on alphabets with at least one isolated point, we show that strong power convergence and quasi-compactness of \(T\) are equivalent and, in fact, that these properties are equivalent to iteration of substitution growing words in length uniformly across all seeds in the alphabet.
\end{abstract}

\maketitle

\section{Introduction}

Core to the field of Dynamical Systems is the symbolic dynamics of subshifts over finite alphabets. In addition to shifts of finite type, key examples of subshifts are given by `self-similar' systems that are generated by symbolic substitution. Here, we consider a natural extension of this setting to allow for infinitely many letters. Let \(\A\) be a compact Hausdorff space, called an alphabet. We consider (generalised) subshifts \(X = X_{\sub}\) over \(\A\), that is, topologically closed and shift-invariant subspaces of \(\A^\Z\), which are generated by a substitution \(\sub \colon \A \to \A^\ast\) that continuously replaces letters in the alphabet with finite words. Such systems have been considered before, for instance in \cite{DOP18}, or the compact automata of \cite[Chapter 12]{Que10}. More general hierarchical systems in the higher dimensional geometric setting can be studied through the fusion formalism established by Frank and Sadun \cite{FS14}. A framework for the study of compact alphabet substitutions was further developed in \cite{MRW25}. Compactness is a natural generalisation of finiteness of the alphabet that retains at least some of the flavour from the finite case, although others have also considered particular examples and classes of substitutions on infinite and untopologised alphabets, see \cite{Fer06,BJPS24}.

An important property of the subshift is unique ergodicity, which roughly pertains to statistical regularity of occurrence of finite words. For finite alphabet substitutions \cite{Fog02,BG13}, primitivity of \(\sub\) implies unique ergodicity of \(X\). Primitivity has a natural extension to compact alphabet substitutions (see \cite{DOP18}, \cite{MRW25} or Section \ref{sec:prereqs}), loosely speaking, it means that iterating substitution \(\sub^n(a)\) on any `seed' letter \(a\) densely (and uniformly, over \(a \in \A\)) fills the alphabet as \(n \to \infty\). Just as for finite substitutions, primitivity implies minimality of \(X\). It is natural to wonder if it still guarantees unique ergodicity.

The answer is no. Durand, Ormes and Petite \cite{DOP18} constructed minimal, self-induced but non-uniquely ergodic Cantor systems, by extension of certain non-uniquely ergodic Toeplitz shifts. A main result there, \cite[Theorem 25]{DOP18}, is that a self-induced minimal Cantor system is conjugate to a generalised subshift generated by a recognisable, primitive, aperiodic substitution on a compact, Hausdorff and zero-dimensional alphabet. The mentioned counterexamples may be defined by constant length substitutions on Cantor alphabets.

There are, however, plenty of situations in which the substitution subshifts are uniquely ergodic. For instance, it is known by the work of the current author with Ma\~{n}ibo and Rust \cite{MRW25} that primitivity \emph{does} imply unique ergodicity in the case of constant length substitutions on alphabets containing an isolated point. Substitutions on alphabets with isolated points can naturally arise, for instance, when compactifying substitutions on infinite, non-topologised alphabets. For bounded length substitutions, this may always be done using the Stone--\v{C}ech compactification, embedding the original alphabet as a set of isolated points, see \cite[Section 3.2]{MRW25}. In practice, much smaller compactifications are usually more natural (and tractable) though, such as the one-point compactification \(\N_\infty \coloneqq \N_0 \cup \{\infty\}\) of the naturals \(\N_0 = \{0,1,\ldots\}\), arguably the simplest case; it is certainly the most `discrete' infinite compact alphabet possible, with just a single accumulation point. Clearly, the topology of \(\N_\infty\) should severely restrict the possible behaviour of any substitution on it, so it is natural to ask: does every primitive substitution on \(\A = \N_\infty\) generate a uniquely ergodic subshift?

In this paper, we show that the answer remains `no'. Consider the substitution on \(\N_0\) given by \(0 \mapsto 1 \mapsto 2\) and \(n \mapsto {0 \ (n-2) \ (n+1)}\) for \(n \geq 2\). For instance, the first few iterations of this on seed \(0\) are:
\[
0 \mapsto 1 \mapsto 2 \mapsto 0 \ 0 \ 3 \mapsto 1 \ 1 \ 0 \ 1 \ 4 \mapsto 2 \ 2 \ 1 \ 2 \ 0 \ 2 \ 5 \mapsto 0 \ 0 \ 3 \ 0 \ 0 \ 3  \ 2 \ 0 \ 0 \ 3 \ 1 \ 0 \ 0 \ 3  \ 0 \ 3 \ 6 \mapsto  \cdots .
\]
This continuously extends to a substitution \(\sub\) on \(\A = \N_\infty\), with \(\sub(\infty) = 0 \infty \infty\) (since \(n-2\) and \(n+1 \to \infty\) as \(n \to \infty\) in the topology of \(\N_\infty\)). As we recall in Section \ref{sec:prereqs}, iteration of substitution defines a language and associated compact (generalised) subshift \(X \subseteq \A^\Z\), which for this example is minimal since the substitution is primitive. However, despite the abundance of isolated points, we prove in Theorem \ref{thm:non-UE example} that \(X\) is not uniquely ergodic. In fact, this substitution does not even admit a natural length function (explained below), answering Questions 4.9 and 6.14 of \cite{MRW25} in the negative, even in the restrictive case of \(\A = \N_\infty\).

For more courteous substitutions, tools to help verify unique ergodicity were developed in \cite{MRW25}. It was shown that for an irreducible substitution (see Section \ref{sec:prereqs}), strong power convergence of an associated operator \(T = (1/r)M\) is sufficient. The `substitution operator' \(M \colon E \to E\) is an analogue of the standard (transposed) substitution matrix \cite{Fog02,BG13} from the finite case, and \(r\) is its spectral radius. Here, \(E = C(\A)\) is the Banach space (with supremum norm) of continuous functions \(f \colon \A \to \R\). Strong power convergence of \(T\) also guarantees (see Lemma \ref{lem:SPC=>NLF}) the existence of a positive eigenvector of \(M\), a so-called natural length function (NLF), which is a continuous map \(\ell \colon \A \to \R_{\geq 0}\) for which, for some \(\lambda \geq 0\) (the `inflation factor') and all \(a \in \A\), we have \(\lambda \ell(a) = \ell(a_1) + \ell(a_2) + \cdots + \ell(a_m)\), where \(\sub(a) = a_1 a_2 \cdots a_m\) for \(a_i \in \A\); for irreducible substitutions (that grow arbitrarily long words), necessarily \(\lambda = r > 1\) and \(\ell\) is a strictly positive function. Such a length function allows one to transform the symbolic rule into a geometric, self-similar one. Here, we show that these results of \cite{MRW25} are tight, by showing that (2) implies (1) in the below:

\begin{theorem}\label{thm:main UE}
If \(\sub\) is an irreducible substitution then the following are equivalent:
\begin{enumerate}
	\item \(T\) is strongly power convergent;
	\item \(\sub\) admits a natural length function and \(X\) is uniquely ergodic.
\end{enumerate}
\end{theorem}

It would be interesting to have a characterisation of just unique ergodicity (with or without a NLF) and it is plausible that this remains equivalent to strong power convergence of \(T\) i.e., that existence of a NLF might follow from unique ergodicity for irreducible substitutions. Of course, we do have that strong power convergence is equivalent to unique ergodicity for irreducible, constant length substitutions, since these have NLF \(\bbo = (a \mapsto 1)\). Generally, the existence of a NLF seems to be closely related to regularity in growth of \(\# \sub^n(a)\) as \(n \to \infty\) over \(a \in \A\). In this direction, unique ergodicity at least forces the growth of \(\# \sub^n(a)\) to have the same exponential part over all \(a \in \A\), see Proposition \ref{prop:superword growth without NLF} (one has tighter control of this growth when a NLF is assumed to exist, see Proposition \ref{prop:UE superword growth}).

The proof of (1) implying (2) of Theorem \ref{thm:main UE} in \cite{MRW25} is to use the NLF to move to the geometric setting of a \(1\)-dimensional tiling space \cite{Sad08} with \(\R\)-action. This is convenient, as it allows one to use Frank and Sadun's fusion framework \cite{FS14} to describe the invariant measures of \(X\) in terms of the action of the dual operator \(T'\) on the positive dual cone (something which also applies to `higher dimensional compact substitutions', such as the pinwheel \cite{Rad94}). For completeness, though, we provide an alternative proof in the Appendix, which remains within the `symbolic setting' here and also allows one to bypass some minor complications (for instance, recognisability is not a concern with this approach).

Even stronger properties of \(T\) are possible. Conditions were found in \cite{MRW25} that ensure \(T\) is quasi-compact which, with primitivity, implies not just strong but even uniform power convergence of \(T\) (to a projection operator mapping to the \(1\)-dimensional subspace spanned by a NLF). Such substitutions behave `more discretely'. And indeed, it was shown that constant length, primitive substitutions with an isolated point satisfy it. The situation is varied in general: there are primitive substitutions with \(T\) uniformly power convergent, others with \(T\) strongly but not uniformly power convergent, and the counterexample above shows that the weakest property of mean ergodicity can also fail to hold (since when \(T\) is mean ergodic, \(\sub\) admits a NLF, see \cite[Section 5]{MRW25}). However, the following shows that this landscape is vastly simplified when \(\A\) contains an isolated point: for primitive substitutions, these properties are equivalent, and also equivalent to uniformity of growth of letters under iteration of substitution:

\begin{samepage}
\begin{theorem}\label{thm:equivalence of operator conditions}
If \(\sub\) is a primitive substitution then each item below implies the next one. Suppose that, additionally, \(\A\) contains an isolated point or, more generally, \(\sub\) has a `weak coincidence', that is, some \(p \in \A\) and \(n \in \N\) for which, for all \(a \in \A\), the letter \(p\) occurs in \(\sub^n(a)\). Then (7) implies (1), so (1--7) are equivalent and, by Theorem \ref{thm:main UE}, all are sufficient for unique ergodicity of \(X\).
\begin{enumerate}
	\item \(T\) is quasi-compact;
	\item \(T\) is uniformly power convergent;
	\item \(T\) is strongly power convergent;
	\item \(T\) is mean ergodic;
	\item \(\sub\) admits a natural length function;
	\item there exist \(\alpha\), \(\beta > 0\) for which, for all \(n \in \N\) and \(a \in \A\), we have \(\alpha r^n \leq \# \sub^n(a) \leq \beta r^n\);
	\item \(\displaystyle \inf_{n \in \N} \frac{ \min_{a \in \A} \# \sub^n(a) }{ \max_{b \in \A} \# \sub^n(b) } > 0\).
\end{enumerate}
\end{theorem}
\end{samepage}

It is perhaps the case that unique ergodicity implies the distortion bound (7), in which case unique ergodicity could be added to the list of equivalent properties. In any case, verifying unique ergodicity cannot be far from establishing any of (1--7) for primitive substitutions with an isolated point (see also Proposition \ref{prop:superword growth without NLF}). Condition (1), of quasi-compactness of \(T\) (equivalently \(M\)) means that the essential spectral radius \(r' \coloneqq r_\mathrm{ess}(M)\) satisfies \(r' < r\). So, in Section \ref{sec:sub reduction}, we further develop methods for bounding \(r'\) from above. They are particularly suited to alphabets with an abundance of isolated points, in particular, with only finitely many accumulation points; in this case, we show that \(r'\) may always be bounded above by the leading eigenvalue of an associated finite substitution on the accumulation points, given by applying a certain reduction procedure on \(\sub\). Perhaps after passing to a power of \(\sub\), this typically yields an exact formula for \(r_\mathrm{ess}(M)\), see in particular Corollary \ref{cor:essential spectral radius} and the examples that follow. Quasi-compactness, and thus unique ergodicity of \(X\), might then be verified for a given example by searching for some \(n \in \N\) with \(\min_{a \in \A} \# \sub^n(a) > r'\) (see, for instance, examples discussed in Section \ref{sec:counterexample discussion}), or proving that, for some \(\epsilon > 0\), one has \(\max_{a \in \A} \# \sub^n(a)  > (r'+\epsilon)^n\) for sufficiently large \(n\).

The paper is organised as follows. In the next section, we review generalised subshifts over compact alphabets, substitutions and the substitution operator. We also recall the uniform version (Oxtoby's Theorem) of Birkhoff's Ergodic Theorem, which characterises unique ergodicity in terms of uniform convergence of time-averages of continuous functions, which we use to prove both directions of Theorem \ref{thm:main UE}. A more self-contained proof to that in \cite{MRW25}, of the first direction of this result, is given in the Appendix for completeness, whilst the new direction is proved in Section \ref{sec:necessary UE conditions}, where we also give some useful results relating unique ergodicity and the existence of a NLF to regularity of growth of words under substitution. In Section \ref{sec:non-UE} we prove, in Theorem \ref{thm:non-UE example}, that the counterexample introduced above has non-uniform growth of superwords across the alphabet. With the results of Section \ref{sec:necessary UE conditions}, this proves the associated system is not uniquely ergodic, and does not have a NLF. Finally, in Section \ref{sec:sub reduction}, we recall the concept of quasi-compactness and develop a notion of substitution reduction with which to effectively estimate the essential spectral radius of the substitution operator for alphabets with an abundance of isolated points. This is demonstrated on a range of examples.  Many of the implications of Theorem \ref{thm:equivalence of operator conditions} appear already in \cite{MRW25}, the main observation here is that they all become equivalent, by (7) implying (1), under the stated conditions; the proof is given in Section \ref{sec:equivalence of operator properties}.

\subsection*{Acknowledgements} The author thanks Lorenzo Sadun for some enlightening observations on this work, discussed during a workshop at the Brin Mathematics Research Center in October 2026.

\section{Prerequisites and notation} \label{sec:prereqs}

We refer the reader to \cite{MRW25}, especially Sections 2--4, for details of most of the following preliminaries. However, we will recall here all of the key concepts and notation required.

\subsection{Compact alphabets and words}
We let \(\N = \{n \in \Z \mid n > 0\}\), and \(\N_0 \coloneqq \{0\} \cup \N\). Throughout, \(\A\) denotes our \textbf{alphabet}, which is a non-empty compact Hausdorff space. We let \(\A^\ast \coloneqq \bigsqcup_{i=0}^\infty \A^i\) denote the space of \textbf{words}, where each \(\A^i\) has the product topology and \(\A^\ast\) is topologised as their disjoint union. Here, \(\A^0 = \{ \varepsilon \}\) is the one-point space consisting of the empty word \(\varepsilon\) which has no letters. We let \(\A^+ \coloneqq \bigsqcup_{i=1}^\infty \A^i = \A^\ast \setminus \{\varepsilon\}\) denote the subspace of non-empty words. Instead of tuples, words \(w = (a_0,a_1,\ldots,a_{n-1}) \in \A^n\) will be spelled out by concatenation i.e., \(w = a_0 a_1 \cdots a_{n-1}\). Concatenation \((a_0 a_1 \cdots a_{n-1},b_0 b_1 \cdots b_{m-1}) \mapsto a_0 a_1 \cdots a_{n-1} b_0 b_1 \cdots b_{m-1}\) defines a continuous map \(\A^\ast \times \A^\ast \to \A^\ast\) where, just as for letters, we denote the concatenation of two words \(u\) and \(v\) by \(uv\). The \textbf{length function} \(\# \colon \A^\ast \to \N_0\), defined by \(\# u = n\) for \(u \in \A^n\), defines a continuous map \(\# \colon \A^\ast \to \N_0\) (where \(\N_0\) is given the discrete topology).

Given \(m \leq n \in \Z \cup \{-\infty,\infty\}\), we let \([m..n)\), \([m..n]\) etc., denote the corresponding \textbf{interval} of (extended) integers i.e., \([m..n) \coloneqq [m,n) \cap (\Z \cup \{-\infty,\infty\})\) e.g., for \(m\), \(n \in \Z\), we have \([m..n) = \{m,m+1,\cdots,n-1\}\). For any finite set \(S\) we denote its number of elements by \(\# S \in \N_0\) which, for an interval \(I\), we call its \textbf{length}.

For \(u \in \A^n\) and \(i \in [0..n)\), we always let \(u_i \in \A\) denote the \(i\)th letter of \(u\), so \(u = u_0 u_1 \cdots u_{n-1}\). Similarly, given an interval \(I = [p..q) \subseteq [0..n)\), we let \(u_I \coloneqq u_p u_{p+1} \cdots u_{q-1} \in \A^\ast\). We write \(v \triangleleft u\), and say that \(v\) is a \textbf{subword} of \(u\), if \(v = u_I\) for some such \(I\); if \(w \in \A^{\Z}\) is a bi-infinite word, \(i \in \Z\) and \(I\) is a finite interval in \(\Z\), we define \(w_i \in \A\), \(w_I \in \A^\ast\) and the subword notation \(v \triangleleft w\) analogously.

\subsection{Compact alphabet substitutions}
A \textbf{substitution} (on \(\A\)) is a continuous map \(\sub \colon \A \to \A^\ast\). Usually, our substitutions will be \textbf{non-erasing}, meaning that \(\sub(\A) \subset \A^+\) i.e., all letters substitute to non-empty words, although it will occasionally be useful to allow for erasing substitutions too. The \textbf{length} of \(\sub\), given by \(L \coloneqq \max_{a \in \A} \# \sub(a) \in \N_0\), is bounded, since \(\# \circ \sub\) is continuous and \(\A\) is compact so \(\#(\sub(\A)) \subset \N_0\) is compact, thus finite. For \(i \in [0..L]\),
\[
\A_i \coloneqq \{ a \in \A \mid \# \sub(a) = i \} .
\]
Since \(\A_i = (\# \circ \sub)^{-1} \{i\}\), and \(\{i\} \subset \N_0\) is clopen (as are all subsets of \(\N_0\)), it follows that \(\A = \A_0 \sqcup \A_1 \sqcup \cdots \sqcup \A_L\) (after removing any \(\A_i = \emptyset\)) is a clopen partition of \(\A\).

By concatenation, substitution naturally extends to a map \(\sub \colon \A^\ast \to \A^\ast\), by \(\sub(a_1 a_2 \cdots a_n) \coloneqq \sub(a_1) \sub(a_2) \cdots \sub(a_n)\) for a word \(a_1 a_2 \cdots a_n \in \A^n\). Thus, \(\sub\) may be iterated on words. A word of the form \(\sub^n(a)\), where \(n \in \N_0\) and \(a \in \A\), is called a \textbf{superword}, or more specifically an \textbf{\(n\)-superword}. A word \(u \in \A^\ast\) is \textbf{generated} (by \(\sub\)) if \(u \triangleleft v\) for some superword \(v\). The topological closure in \(\A^\ast\) of the subset of generated words is denoted \(\lang(\sub)\), called the `legal words' or \textbf{language} of \(\sub\). Intuitively, a word \(u\) is legal when it is either generated, or there is a sequence of \(n\)-superwords for increasing \(n\) which contain subwords that, letter-wise, are smaller and smaller perturbations of \(u\).

The language defines the subshift associated to \(\sub\), given by
\[
X_{\sub} \coloneqq \{ w \in \A^\Z \mid \text{if } u \in \A^\ast \text{ and } u \triangleleft w \text{ then } u \in \lang(\sub)\} ,
\]
which is given the subspace topology in the \textbf{full shift} \(\A^\Z\), which has the usual product topology (so is compact, by Tychonoff's Theorem). Throughout, we drop the subscript and write \(X_{\sub} = X\). We let \(\sigma \colon \A^\Z \to \A^\Z\) denote the left-shift \(\sigma((w_i)_i) \coloneqq (w_{i+1})_i\), which is a homeomorphism. It can be shown that \(X\) is also compact and \(\sigma\)-invariant in the full shift, that is, \(\sigma\) restricts also to a homeomorphism \(\sigma \colon X \to X\). We have \(X \neq \emptyset\) precisely when \(\sub\) generates arbitrarily long words, that is, when for all \(k \geq 0\) there exists some \(a \in \A\) and \(n \in \N_0\) with \(\# \sub^n(a) \geq k\), or equivalently every \(\lang(\sub) \cap \A^k \neq \emptyset\). We make it a running assumption that this is the case. Just as for finite substitutions, elements \(w \in X\) are `hierarchical': the letters of any \(w \in X\) may be grouped into a sequence of superwords that also define an element of \(X\). More precisely, for all \(w \in X\), there exists some \(w' \in X\) for which \(w\) is a shift of \(\sub w'\), specifically, \(w = \sigma^j(\sub w')\) for some \(j \in [0..\# \sub(w'_0))\).

A substitution \(\sub\) is \textbf{irreducible} if it cannot be restricted to a smaller compact alphabet substitution. That is, there does not exist a non-empty closed subset \(\mathcal{B} \subseteq \A\) for which \(\sub(\mathcal{B}) \subseteq \mathcal{B}^\ast\). A stronger (but related) property is primitivity, which is a natural extension of the same notion from finite alphabet substitutions. We say that \(\sub\) is \textbf{primitive} if, for all non-empty open subsets \(U \subseteq \A\) (have in mind some neighbourhood of a letter), there exists some \(n = n_U\) so that all \(n\)-superwords contain a letter in \(U\). Primitivity implies minimality of the dynamical system \((X,\sigma)\) i.e., \(X\) does not contain any non-empty, closed and \(\sigma\)-invariant subset (equivalently, all \(x \in X\) have dense orbit under application of the shift). Minimality is equivalent to a natural extension of the standard notion of `repetitivity' (uniform recurrence) of words to such generalised subshifts, see \cite[Proposition 2.14]{MRW25}; roughly speaking, all finite words, up to arbitrarily small perturbation of the letters, appear with bounded gaps in all \(w \in X\).

\subsection{Unique ergodicity}
A further desirable property to minimality is not just `bounded gaps', but also statistical regularity of recurrence of words i.e., unique ergodicity of \((X,\sigma)\), which we recall below. In the finite alphabet case, primitivity implies unique ergodicity, but it is shown in \cite{DOP18} that there are counterexamples to this when \(\A\) is the Cantor set. To prevent unnecessary restrictions on our alphabets (as well as the introduction of distracting measure theoretic technicalities), measures will be identified with linear functionals, as follows. We let \(C(X)\) denote the Banach space of continuous functions \(f \colon X \to \R\) with the supremum norm \(\|f\| \coloneqq \sup_{x \in X} |f(x)|\). In fact, it is a Banach lattice \cite{Sch74}, when equipped with an ordering defined by the \textbf{positive cone}
\[
K(X) \coloneqq \{f \in C(X) \mid f(x) \geq 0 \text{ for all } x \in X \} .
\]
For \(f\), \(g \in C(X)\), we write \(f \leq g\) if \(g-f \in K(X)\). We then have the Banach dual \(C(X)'\) of continuous linear maps \(\nu \colon C(X) \to \R\), and its dual positive cone
\[
K'(X) \coloneqq \{\nu \in C(X)' \mid \nu(f) \geq 0 \text{ for all } f \in K \} .
\]
An element \(\nu \in K'(X)\) is called a \textbf{measure}. Thus, as one tends to hope, a measure is precisely a way of continuously and linearly converting continuous functions into scalars, their `integrals', and in such a way that integrals of non-negative functions are non-negative. We say that a measure \(\mu\) is a \textbf{probability measure} if \(\mu(\bbo) = 1\), where \(\bbo \colon X \to \R\) is the constant function \(\bbo(x) = 1\) for all \(x \in X\). By the Riesz--Markov--Kakutani Representation Theorem, we are thus restricting to positive, regular Borel probability measures on \(X\). Its shift \(\sigma \mu\) is the measure defined by \((\sigma \mu)(f) \coloneqq \mu(\sigma f)\), and \(\mu\) is \textbf{\(\sigma\)-invariant} if \(\mu = \sigma \mu\). Finally, we say that \(X\) is \textbf{uniquely ergodic} if it has a unique invariant probability measure.

Aside from some notational advantages in our current context, this formalism means that metrisability need not be assumed in the standard characterisation of unique ergodicity (Birkhoff's Ergodic Theorem, or more specifically Oxtoby's uniform version for uniquely ergodic systems), given below. For a finite set (almost always an interval) \(I \subset \Z\), \(f \in C(X)\) and \(w \in X\), we denote the corresponding \textbf{time-average} by
\[
\avg^I_w(f) \coloneqq \frac{1}{\# I} \sum_{j \in I} f(\sigma^j w) .
\]
When the function \(f\) is understood, we may write simply \(\avg_w^I\), and in the standard case that \(I = [0,n)\) we simplify further, to \(\avg_w^n\). For an interval \(I = [m,n)\), we have \(\avg_w^I = \avg_{\sigma^m w}^{n-m}\), so of course it would be sufficient (and more standard) to only consider time-averages over intervals of the form \([0,n)\), although we prefer the slight notational flexibility here, when splitting time-average calculations; given any partition \(I = I_1 \sqcup \cdots \sqcup I_k\), we have the trivial identity
\begin{equation} \label{eq:split average}
\avg^I_w = \frac{\# I_1}{\# I} \avg^{I_1}_w + \cdots + \frac{\# I_k}{\# I} \avg^{I_k}_w .
\end{equation}

Given \(x\), \(y \in \R\) and \(\epsilon > 0\), we will write \(x = y \pm \epsilon\) if \(|x-y| \leq \epsilon\). Unique ergodicity can then be characterised (see, for instance, \cite[Theorem 10.6]{EFHN15}) as follows:

\begin{theorem} \label{thm:UE characterisation}
For a compact Hausdorff space \(X\) and \(\sigma \colon X \to X\), the following are equivalent:
\begin{enumerate}
	\item time-averages of all continuous functions converge uniformly. That is, for all \(f \in C(X)\), there exists \(c(f) \in \R\) so that, for all \(\epsilon > 0\) and \(w \in X\), there exists some \(k = k(\epsilon) \in \N\) and \(c(f) \in \R\) so that, for all \(n \geq k\), we have \(\avg_w^n(f) = c(f) \pm \epsilon\);
	\item \((X,\sigma)\) is uniquely ergodic, that is, there exists a unique \(\sigma\)-invariant \(\nu \in K'\) with \(\nu(\bbo) = 1\).
\end{enumerate}
The time-average \(c(f)\) from (1) is then necessarily equal to the space-average \(\nu(f)\) from (2).
\end{theorem}

A letter \(a \in \A\) is called \textbf{isolated} if \(\{a\}\) is open, otherwise it is called an \textbf{accumulation point}. If \(a\) is isolated then the function \(f \colon X \to \R\), given by \(f(w) = 1\) if \(w_0 = a\) and \(f(w) = 0\) otherwise, is continuous (since it is the composition \(\chi_a \circ \pi_0\) of the projection \(\pi_0 = w \mapsto w_0 \in \A\) followed by the continuous characteristic function \(\chi_a \colon \A \to \R\)). Thus, when \(X\) is uniquely ergodic, \(a\) appears with a well-defined `frequency' \(\nu(f)\) to which the time-averages converge uniformly across \(w \in X\). Generally, we can imagine calculation of averages for other continuous quantities on \(X\), those that take on similar values when words agree to a large distance about index \(0\), up to small perturbations of the letters.

\subsection{The substitution operator} \label{sec:sub operators}
Finally, we must introduce the substitution operator, which plays a similar role to the (transpose of) the substitution matrix from the finite alphabet setting. Throughout, we let \(E \coloneqq C(\A)\) denote the Banach lattice of continuous functions \(f \colon \A \to \R\) (as above, with supremum norm), and \(K\) its positive cone of functions \(f(a) \geq 0\) for all \(a \in \A\). We let \(K_{>0} \subset K\) denote those functions with \(f(a) > 0\) for all \(a \in \A\). The \textbf{substitution operator} \(M \colon E \to E\) is given by
\[
(M f)(a) \coloneqq \sum_{b \triangleleft \sub(a)} f(a) ,
\]
where here (and generally) a sum as the above is over all letters of \(\sub(a)\) \emph{with multiplicity}. For example, for the periodic-doubling substitution with \(\A = \{a,b\}\), \(\sub(a) = ab\) and \(\sub(b) = aa\), given \(f \in E\) the new function \(Mf \colon \A \to \R\) is given by \(a \mapsto f(a) + f(b)\) and \(b \mapsto 2 f(a)\).

Clearly, \(M\) is linear and positive (\(M(K) \subseteq K\)). It is also not difficult to show \cite[Lemma 4.5]{MRW25} that for any \(k \in \N_0\), the substitution operator of \(\sub^k \colon \A \to \A^\ast\) is given by \(M^k\). By positivity, \(\|M^k\| = \|M^k(\bbo)\| = \max_{a \in \A} \# \sub^k(a) < \infty\), and \(M\) is bounded. By Gelfand's Formula, the spectral radius \(r \coloneqq r(M)\) of \(M\) is given by
\[
r \coloneqq r(M) = \lim_{n \to \infty} \sqrt[n]{\| M^n \|} = \inf_{n \to \infty} \sqrt[n]{\| M^n \|} = \inf_{n \to \infty} \sqrt[n]{ \max_{a \in \A} \# \sub^n(a) } .
\]
It is not hard to show that, for all \(n \in \N\), we have the lower bound \( r^n \geq \min_{a \in \A} \# \sub^n(a)\) (see \cite[Lemma 4.23]{MRW25}). We let \(T \coloneqq \frac{1}{r} M\) denote the \textbf{normalised substitution operator}.

It should be noted that, of course, \(M\) does not depend on the ordering of letters in the \(1\)-superwords defining a substitution. That is, it only depends on substitution considered as a mapping of finite multisets (rather than ordered words), or the `Abelianisation' of \(\sub\), which may be connected more directly and explicitly to the dual operator \(M' \colon E' \to E'\). Here, \(E'\) is the continuous dual of \(E\), that is, the Banach space of continuous linear maps \(\nu \colon E \to \R\), with norm \(\|\nu(f)\| \coloneqq \sup_{\|f\| \leq 1} |\nu(f)|\), and \(M'\) is defined by \( (M'\nu)(f) \coloneqq \nu(M f)\). For \(a \in \A\), consider the associated Dirac delta \(\delta_a \in E'\), given by \(\delta_a(f) \coloneqq f(a)\). Suppose we defined a map on the set of these Dirac-deltas where, given \(a \in \A\) with \(\sub(a) = a_1 a_2 \cdots a_n \in \A^n\), we let
\[
\delta_a \mapsto \delta_{a_1} + \delta_{a_2} + \cdots + \delta_{a_n} .
\]
This topologically realises the `Abelianisation' of \(\sub\) as a map of multisets i.e., where each \(a \in \A\) maps to the unordered set of \(a_i\), with multiplicities. Then \(M'\) is the natural linearisation and topological completion of this map to all of \(E'\), and an analogue of the standard substitution matrix from the finite alphabet setting. And just as for the finite case, the positive eigenvectors are closely related to the invariant measures of the system \(X\), see \cite[Section 5]{MRW25} or the Appendix. Regularity properties of \(T\) will be important here, such as strong power convergence, but we delay introduction of these until they are needed.

\section{Necessary conditions for unique ergodicity} \label{sec:necessary UE conditions}

A function \(\ell \in K\) is called a \textbf{natural length function} (or \textbf{NLF}, for short) if \(\ell \neq 0\) and there exists some \(\lambda \geq 0\) with \(M \ell = \lambda \ell\). That is, \(\ell\) is a positive eigenvector of \(M\). When \(\sub\) is irreducible, necessarily \(\lambda = r\) i.e., \(M \ell = r\ell\) or, dividing by \(r\), we have \(T \ell = \ell\), that is, \(\ell\) is a fixed point of \(T\). The simple proof of this, and the following, can be found in \cite[Theorem 4.26]{MRW25}.

\begin{lemma}\label{lem:NLF unique and positive}
Suppose that \(\sub\) is irreducible and has natural length function \(\ell \in K\). Then \(\ell \in K_{>0}\), is a fixed point of \(T\) and every other natural length function is a scalar multiple of \(\ell\). Moreover, \(\ell\) is the only fixed point of \(T\) (whether positive or not), up to rescaling.
\end{lemma}

Recall that we assume throughout that \(\sub\) generates arbitrarily large words, that is, for all \(k \in \N\) there exists some \(n \in \N\) and \(a \in \A\) with \(\# \sub^n(a) \geq k\). This does not generally imply that we can find \(n\) large enough so that \(\# v \geq k\) for all \(n\)-superwords, indeed, it is even possible that every sequence \(\# \sub^n(a)\) is bounded, for any \emph{fixed} letter \(a \in \A\) \cite[Example 3.12]{MRW25}. However, irreducibility assures that all superwords grow together, in the following sense, see \cite[Proposition 4.27]{MRW25}.

\begin{lemma}\label{lem:irreducible=>letter growth}
If \(\sub\) is irreducible then, for all \(k \in \N\), there exists some \(n \in \N\) so that, for all \(a \in \A\) we have \(\# \sub^n(a) \geq k\).
\end{lemma}

Define the language \(\lang(X) \subseteq \A^\ast\) of the subshift to be the set of finite words that appear, that is \(\lang(X) \coloneqq \{u \in \A^\ast \mid u \triangleleft w \text{ for some } w \in X\}\). By definition, \(\lang(X) \subseteq \lang(\sub)\). Subshifts defined by irreducible substitutions have `full' languages:

\begin{lemma}
If \(\sub\) is irreducible then \(\lang(X) = \lang(\sub)\).
\end{lemma}

\begin{proof}
A compactness argument \cite[Proposition 3.16]{MRW25} shows it is sufficient that, for all \(n \in \N\) and all non-empty open subsets \(U \subseteq \A\), there is a word of length \(2n+1\) generated by \(\sub\) whose central letter is in \(U\). This holds for irreducible substitutions. Indeed, first take any generated word \(u b v \in \A^{2n+1}\), where \(u\), \(v \in \A^n\) and \(b \in \A\) (which exists, by Lemma \ref{lem:irreducible=>letter growth}). Now, for some sufficiently large \(N \in \N\) we have \(\sub^N(b)\) contains a letter of \(U\). Indeed, otherwise, consider its `orbit'
\[
O \coloneqq \{ a \in \A \mid a \triangleleft \sub^i(b) \text{ for some } i \in \N_0\} .
\]
Clearly, \(O\) is closed under substitution (that is, \(\sub(O) \subseteq O\)), and thus so is its topological closure \(\mathcal{O} \coloneqq \overline{O}\), by continuity. Moreover, since \(O\) is a subset of the closed subset \(\A \setminus U\), the same is true of \(\mathcal{O}\), so \(\mathcal{O} \neq \A\) (and obviously \(\mathcal{O} \neq \emptyset\)), contradicting irreducibility. So, for some \(N \in \N_0\) and \(a \in U\) we have \(a \triangleleft \sub^N(b)\). Since \(\# \sub^N(u)\), \(\# \sub^N(v) \geq n\) (irreducible substitutions are not erasing), the generated word \(\sub^N(ubv) = \sub^N(u) \sub^N(b) \sub^N(v)\) contains a length \(2n+1\) word centred on \(a \in U\).
\end{proof}

Since every superword is generated, and thus legal, in particular we have the following:

\begin{corollary}\label{cor:language contains superwords}
If  \(\sub\) is irreducible then, for all \(a \in \A\) and \(n \in \N_0\), we have \(\sub^n(a) \in \lang(X)\), so there exists some \(w \in X\) with \(w_I = \sub^n(a)\), where \(I = [0,\# \sub^n(a))\).
\end{corollary}

We now give some simple growth properties of superwords that must hold if \(X\) is uniquely ergodic, firstly when \(\sub\) admits a NLF, and then a weaker result when this assumption is dropped.

\begin{proposition}\label{prop:UE superword growth}
Suppose that \(\sub\) is irreducible, admits a natural length function \(\ell\) and that \(X\) is uniquely ergodic. Then, following a suitable rescaling of \(\ell\), for all \(\epsilon > 0\) there exists some \(N = N(\epsilon) \in \N\) so that, for all \(n \geq N\),
\begin{equation}
\frac{r^n \ell(a)}{ \# \sub^n(a) } = 1 \pm \epsilon .
\end{equation}
\end{proposition}

\begin{proof}
Let \(\epsilon > 0\) and define \(f \colon X \to \R\) by \(f(w) \coloneqq \ell(w_0)\), which is continuous by continuity of \(\ell\) (and the projection \(w \mapsto w_0 \in \A\)). Since \(\inf_{a \in \A} \ell(a) > 0\), by irreducibility (Lemma \ref{lem:NLF unique and positive}), clearly \(c = c(f) > 0\) in Theorem \ref{thm:UE characterisation}. Dividing \(\ell\) by \(c\) if necessary (which is obviously still a NLF), we may assume that \(f\) has time-averages uniformly converging to \(1\), say within \(\epsilon\) of \(1\) for all intervals of length at least \(k = k(\epsilon)\). By Lemma \ref{lem:irreducible=>letter growth}, there exists \(N \in \N\) for which, for all \(n \geq N\) and \(a \in \A\), we have \(\# \sub^n(a) \geq k\).

By Corollary \ref{cor:language contains superwords}, for any \(n \in \N\) and \(a \in \A\), there exists some \(w_a \in X\) with \((w_a)_{[0,\# \sub^n(a))} = \sub^n(a)\). Then the time-average characterisation of unique ergodicity gives, as required,
\begin{equation}
\frac{1}{\# \sub^n(a)} [r^n \ell(a)] = \frac{1}{\# \sub^n(a)} [(M^n \ell)(a)] = \frac{1}{\# \sub^n(a)} \left[ \sum_{b \triangleleft \sub^n(a)} \ell(b) \right] = \avg_{w_a}^{\# \sub^n(a)}(f) = 1 \pm \epsilon .
\end{equation}
\end{proof}

As we saw in Theorem \ref{thm:main UE}, it would be incredibly useful to know if irreducibility with unique ergodicity implies the existence of a NLF. Even if not, we may at least show the base of the exponential part of the growth has constant \(r\) across the alphabet:

\begin{proposition}\label{prop:superword growth without NLF}
Suppose that \(\sub\) is irreducible and \(X\) is uniquely ergodic. Then, for all \(0 < \epsilon < r\), there exists some \(N = N(\epsilon)\) so that, for all \(a \in \A\) and \(n \geq N\),
\[
(r-\epsilon)^n \leq \# \sub^n(a) \leq (r+\epsilon)^n .
\]
\end{proposition}

\begin{proof}
For \(i \in [1..L]\), consider the indicator functions \(f_i \colon \A \to \R\) given by \(f_i(x) = 1\) for \(x \in \A_i\) and \(f_i(x) = 0\) otherwise. Since \(\A = \A_1 \sqcup \cdots \sqcup \A_L\) is a partition into clopen subsets, each \(f_i\) is continuous. Each \(f_i\) canonically lifts to a continuous function \(F_i \colon X \to \R\), that is, \(F_i(w) \coloneqq f_i(w_0)\). Thus, by Theorem \ref{thm:UE characterisation}, with respect to the unique invariant measure \(\nu\), the frequencies \(\nu(F_i) \coloneqq p_i\) are well defined and the time-averages \(\avg_w^n(F_i)\) converge uniformly to \(p_i\). Define \(\alpha \coloneqq p_1 + 2p_2 + \cdots + Lp_L\).

By an analogous argument to in the previous result, given \(0 < \epsilon < \alpha\), there exists \(N\) so that, for all \(a \in \A\) and \(n \geq N\), we have
\begin{equation} \label{eq:Ai frequency estimates}
\frac{\# \{b \triangleleft \sub^n(a) \mid b \in \A_i\}}{\# \sub^n(a)} = \frac{1}{\# \sub^n(a)} \sum_{b \triangleleft \sub^n(a)} f_i(b) = p_i \pm \frac{2\epsilon}{L(L+1)} .
\end{equation}
Therefore, we obtain bounds
\[
\# \sub^n(a) \left(p_i - \frac{2\epsilon}{L(L+1)} \right) \leq \# \{b \triangleleft \sub^n(a) \mid b \in \A_i\} \leq \# \sub^n(a) \left(p_i + \frac{2\epsilon}{L(L+1)} \right) .
\]
By definition, each \(b \in \A_i\) substitutes to \(i\) letters. Thus, given some \(a \in \A\), summing the above inequalities over each \(i \in [1..L]\) we obtain
\[
\sum_{i=1}^L i \cdot \#\sub^n(a)\left( p_i - \frac{2\epsilon}{L(L+1)} \right) \leq \# \sub^{n+1}(a) \leq \sum_{i=1}^L i \cdot \# \sub^n(a) \left( p_i + \frac{2\epsilon}{L(L+1)} \right) .
\]
Evaluating the upper and lower limits, \(\# \sub^n(a)(\alpha - \epsilon) \leq \# \sub^{n+1}(a) \leq \# \sub^n(a)(\alpha + \epsilon)\) for all \(n \geq N\). Iterating this from \(n = N\), we obtain, for all \(i \geq 0\),
\[
\# \sub^N(a) (\alpha-\epsilon)^i \leq \# \sub^{N + i}(a) \leq \# \sub^N(a) (\alpha + \epsilon)^i .
\]
Since \(\# \sub^N(a) \geq 1\), denoting \(C = C(\epsilon) \coloneqq \max_{a \in \A} \# \sub^N(a)\) we have, for all \(a \in \A\) and \(i \geq 0\),
\begin{equation}\label{eq:growth before spectral radius}
(\alpha-\epsilon)^i \leq \# \sub^{N + i}(a) \leq C(\alpha + \epsilon)^i .
\end{equation}
So we only need to show \(\alpha = r\). Supposing that \(\alpha > r\), taking \(\epsilon = (\alpha-r)/2 > 0\) the above implies that, for any \(a \in \A\),
\[
\|M^{N+i}\| \geq \# \sub^{N+i}(a) \geq \left(\alpha - \frac{\alpha-r}{2} \right)^i = (r+\epsilon)^i 
\]
so, taking roots, \( \|M^{N+i}\|^{1/(N+i)} \geq (r+\epsilon)^{i/(N+i)}\), contradicting Gelfand's Formula. Then, by Gelfand's Formula, \(r = \lim_i \|M^{N+i}\|^{1/(N+i)} \geq \lim_i (r+\epsilon)^{i/(N+i)} = r+\epsilon\), a contradiction. Taking \(\alpha < r\) leads to a similar contradiction. Indeed, suppose that \(\alpha < r\) and set \(\epsilon = (r-\alpha)/2\), so that \(\alpha + \epsilon = r - \epsilon\). Inserting this into Equation (\ref{eq:growth before spectral radius}) gives \(\# \sub^{N + i}(a) \leq C(r - \epsilon)^i\). In particular, taking \(a \in \A\) maximising \(\# \sub^{N+i}(a)\), it follows that \(\|M^{N+i}\| \leq C(r - \epsilon)^i\). Taking roots and limits, \(r = \lim_i \|M^{N+i}\|^{1/(N+i)} \leq \lim_i C^{1/(N+i)} (r-\epsilon)^{i/(N+i)} = r-\epsilon\), a contradiction. Thus, \(\alpha = r\) in Equation (\ref{eq:growth before spectral radius}). By increasing \(N(\epsilon)\) if necessary, the multiplicative constant can be absorbed into the exponential term, as required.
\end{proof}

As a corollary of the above, if there exist \(a\), \(b \in \A\) with \(\# \sub^n(a)\) and \(\# \sub^n(b)\) growing with different exponential rates under an irreducible substitution, then \(X\) is not uniquely ergodic. This will be exploited later, in proving the existence of a primitive substitution on \(\A = \N_\infty\) with \(X\) not uniquely ergodic.

To conclude this section, we now prove the new direction of Theorem \ref{thm:main UE} from the introduction, which states that if \(\sub\) admits a NLF and \(X\) is uniquely ergodic, then \(T\) is strongly power convergent. The latter means that, for all \(f \in E\), we have that \(T^n f\) converges in \(E\) as \(n \to \infty\), with respect to the supremum norm. That is, there exists some \(g \in E\) for which, for all \(\epsilon > 0\), there exists some \(N \in \N\) so that, for all \(n \geq N\), we have \(\|T^n f - g\| \leq \epsilon\). In the process of the proof, we find that \(g\) is necessarily a multiple of the NLF.

\begin{proof}[Proof of (2) implies (1) in Theorem \ref{thm:main UE}]
As already noted, that (1) implies (2) is shown in \cite{MRW25}; the assumption of primitivity in \cite[Theorem 1.1]{MRW25} can indeed be weakened to irreducibility (by combining Theorem 5.22 and Corollary 5.23 there), although we are not aware of examples that are irreducible, uniformly ergodic but not primitive. A more direct proof of this result is given in the Appendix.

So, let \(X\) be uniquely ergodic, have a NLF and \(f \in E\) be arbitrary; we aim to show that \(T^n(f)\) is convergent in \(E\), so let \(\epsilon > 0\) be arbitrary. Also let \(\delta = \delta(\epsilon) > 0\), which can be defined at this stage explicitly in terms of \(\epsilon\) and \(\sub\) (but we give that formula later). Without loss of generality, by scaling \(f\), we may assume that \(\|f\| \leq 1\). Define \(F \colon X \to \R\) as its canonical lift i.e., \(F(w) \coloneqq f(w_0)\) and \(c \coloneqq \nu(f)\) where \(\nu\) is the unique invariant probability measure on \(X\).

By definition of the operators,
\begin{equation}\label{eq:UE1}
(T^n f)(a) = \frac{1}{r^n}(M^n(f))(a) = \frac{1}{r^n} \sum_{b \triangleleft \sub^n(a)} f(b) = \left[ \frac{ \# \sub^n(a)}{r^n} \right] \frac{1}{ \# \sub^n(a) } \sum_{b \triangleleft \sub^n(a)} f(b) .
\end{equation}
By Proposition \ref{prop:UE superword growth}, for the natural scaling of the NLF, the following holds for sufficiently large \(n\)
\begin{equation}\label{eq:UE2}
1 = \frac{r^n \ell(a)}{ \# \sub^n(a)} \pm \delta .
\end{equation}
By strict positivity of \(\ell\) and compactness, there are constants \(c_1\) and \(c_2\) with \(c_1 \bbo \leq \ell \leq c_2\bbo \). Thus, \(c_1 |f| \leq c_1 \bbo \leq \ell\) and so, for all \(n \in \N\),
\[
c_1 |T^n f| \leq T^n(c_1 |f|) \leq T^n(\ell) = \ell \leq c_2 \bbo .
\]
Hence, for all \(a \in \A\), we have \(|(T^n f)(a)| \leq c_2 / c_1\). It follows that
\begin{equation}\label{eq:UE3}
(T^n f)(a) = \ell(a) \cdot \left( \frac{1}{ \# \sub^n(a) } \sum_{b \triangleleft \sub^n(a)} f(b) \right) \pm \frac{c_2}{c_1} \delta ,
\end{equation}
by multiplication of Equations (\ref{eq:UE1}) and (\ref{eq:UE2}).

Now, take \(k\) sufficiently large that time-averages of \(F\) over all intervals of length \(k\) are equal to \(c \pm \delta\). We may then find \(N \in \N\) so that, for all \(n \geq N\), Equation (\ref{eq:UE3}) holds and also \(\# \sub^n(a) \geq k\) for all \(a \in \A\). By Lemma \ref{cor:language contains superwords}, for all \(a \in \A\) and \(n \geq N\), we may find some \(w_a \in X\) with \(w_{[0,\# \sub^n(a))} = \sub^n(a)\). For such a word, the time-average \(\avg_{w_a}^{\# \sub^n(a)}(F) = c \pm \delta\) is exactly equal to the average bracketed on the right-hand side of Equation (\ref{eq:UE3}) so, since \(\ell(a) \leq c_2\), we have
\[
(T^n(f))(a) = c \cdot \ell(a) \pm \delta \left( \frac{c_1}{c_2} + c_2 \right) .
\]
Hence, by originally taking \(\delta = \delta(\epsilon) = \epsilon/((c_1/c_2)+c_2)\), this shows \(\|T^n(f) - c\ell\| \leq \epsilon\) for all \(n \geq N\). Since \(f \in E\) and \(\epsilon > 0\) were arbitrary, \(T\) is strongly power convergent.
\end{proof}

\section{A primitive substitution on \(\N_\infty\) that is not uniquely ergodic} \label{sec:non-UE}

Recall the following substitution on \(\A = \N_\infty\) from the introduction:
\[
\sub =
\begin{cases}
0 \mapsto 1 \mapsto 2 ,       & \\
a \mapsto 0 \ (a-2) \ (a+1) & \text{for } a \geq 2 ,
\end{cases}
\]
where continuity (and the topology of \(\N_\infty\)) dictates that \(\sub(\infty) = 0 \infty \infty\). The idea is simple: substitutions of finite letters have enough terms in the doldrums of the slow-growing \(0\) and \(1\) letters that words fail to grow at the same rate as \(\infty\), which we prove in the following:

\begin{theorem}\label{thm:non-UE example}
For the above substitution, for all \(a \in \N_0\) there exists a constant \(c = c_a \geq 1\) for which, for all \(n \in \N_0\), we have
\[
\# \sub^n(a) \leq c \alpha^n
\]
where \(\alpha = 39/20 = 1.95 < 2\). In contrast, \(\# \sub^n(\infty) \geq 2^n\). The substitution is primitive. Thus, it provides an example of a primitive substitution on \(\N_\infty\) that does not admit a natural length function and has non-uniquely ergodic subshift \(X\).
\end{theorem}

\begin{proof}
Beginning with the simpler claims, since under each iteration \(\sub\) doubles the number of copies of \(\infty\), it is obvious that \(\# \sub^n(\infty) \geq 2^n\) for all \(n \in \N_0\). To show primitivity (in fact, even for the substitution without the \(a-2\) term), first note that, for any letter \(a \in \A\), we have \(a+1 \triangleleft \sub(a)\), thus for all \(k \in \N_0\) we have \(a+(2+k) \triangleleft \sub^{2+k}(a)\). Since \(a+2+k \in [2..\infty]\), applying substitution a further time gives \(0 \triangleleft \sub^{3+k}(a)\), and another \(i\) times therefore gives \(i \triangleleft \sub^{3+k+i}(a)\). Taking \(i = 0\), \(1\) and \(2\), and \(k = 2-i\) for each, it follows that \(i \triangleleft \sub^5(a)\), that is, for all \(a \in \A\), we have \(\sub^5(a)\) contains at least one occurrence of \(0\), \(1\) and \(2\). Then \(\sub^{5+k}(a)\) covers \([0..2+k]\), since under each iteration every letter in the previous step gives an occurrence of its successor, and we also recover a \(0\) letter from the substitution of \(2\) (or any letter bigger than \(2\)).

Primitivity follows. Indeed, let \(U \subseteq \A\) be non-empty and open. Thus, \(U\) contains some \(m \in \N_0\). By the above, for all \(a \in \A\), we have \(m \triangleleft \sub^{5+m}(a)\). That is, all \((5+m)\)-superwords contain a letter in \(U\), as required. If the first claim of the theorem is established, it follows that \(\sub\) is primitive yet \(X\) is not uniquely ergodic, by the contrapositive of Proposition \ref{prop:superword growth without NLF}, since letters in \(\N_0\) have a different exponential growth rate to \(\infty\) under substitution. It also cannot have a NLF, since this would imply that there are constants with \(c_1 r^n \leq \# \sub^n(a) \leq c_2 r^n\) for all \(a \in \A\), see the proof of (5) implying (6) in Theorem \ref{thm:equivalence of operator conditions} (at the end of Section \ref{sec:sub reduction}).

To prove that finite letters have slower growth rate, we define a `score' \(s(a)\) for each finite letter \(a \in \N_0\) recursively, as follows: \(s(0) = 1\), \(s(1) = \alpha\), \(s(2) = \alpha^2\) and, for \(a \geq 2\), we let \(s(a+1) = \alpha s(a) - s(a-2) - 1\). For a finite word \(w \in (\N_0)^\ast\), we define its score \(s(w)\) to be the sum of scores of its letters (of course, with multiplicity). For all \(a \in \N_0\), we have \(s(\sub(a)) = \alpha s(a)\); in fact, this formula informed how we defined the recursion. Indeed, \(s(\sub(0)) = s(1) = \alpha = \alpha s(0)\) and \(s(\sub(1)) = s(2) = \alpha^2 = \alpha s(1)\). For \(a \geq 2\) we have \(s(\sub(a)) = s(0) + s(a-2) + s(a+1) = s(0) + s(a-2) + (\alpha s(a) - s(a-2) - s(0)) = \alpha s(a)\). It follows that, for any finite word \(w = a_1 a_2 \cdots a_m \in (\N_0)^\ast\), we have \(s(\sub(w)) = \alpha s(w)\), since
\[
s(\sub(w)) = \sum_{a_i \triangleleft w} s(\sub(a_i)) = \sum_{a_i \triangleleft w}  \alpha s(a_i) = \alpha \sum_{a_i \triangleleft w} s(a_i) = \alpha s(w) .
\]
Then by induction, for all \(a \in \A = \A^1 \subseteq \A^\ast\) and \(n \in \N_0\), we have \(s(\sub^n(a)) = \alpha^n s(a)\).

We now claim that \(s(a) \geq 1\) for each \(a \in \N_0\). In fact, we show that for sufficiently large \(a\), we have \(s(a) \geq 20\) and \(s(a+1) \geq \beta s(a)\), where \(\beta = 5/4\). Supposing we can find some particular \(a = k \geq 2\) satisfying
\begin{equation} \label{eq:score increases}
s(a) \geq 20, \ s(a) \geq \beta s(a-1) \ \text{and} \ s(a-1) \geq \beta s(a-2) 
\end{equation}
it would follow that
\[
s(k+1) = \alpha s(k) - s(k-2) - 1 \geq s(k) \left( \alpha - \beta^{-2} - \frac{1}{20} \right) .
\]
Since \(\alpha - \beta^{-2} - 1/20 = 63/50 \geq 252/200 \geq 250/200 = \beta\), Equation (\ref{eq:score increases}) also holds for \(a = k+1\), as required for induction, so we just need to show Equation (\ref{eq:score increases}) holds for some \(a = k \geq 2\). It turns out that \(k = 8\) is sufficient; explicitly the first eight scores have the following values:
\begin{align*}
s(0) = 1 , \ s(1) = \frac{39}{20} , \ s(2) = \frac{\ths{1521}}{400} , s(3) = \frac{\ths{43319}}{8000} , \ s(4) = \frac{\ths{1217441}}{\ths{160000}} , \ s(5) = \frac{\ths{32112199}}{\ths{3200000}} , \\ 
s(6) = \frac{\ths{841823761}}{\ths{64000000}} , \ s(7) = \frac{\ths{21811598679}}{\ths{1280000000}} , \ s(8) = \frac{\ths{568154756481}}{\ths{25600000000}} \geq 20 
\end{align*}
where the pertinent ratios are
\[
\frac{s(8)}{s(7)} = \frac{\ths{14568070679}}{\ths{11185435220}} \geq \frac{5}{4} \ \text{and} \ \frac{s(7)}{s(6)} = \frac{\ths{21811598679}}{\ths{16836475220}} \geq \frac{5}{4} ,
\]
as required. Let us note that, to avoid such calculations, we could have chosen \(s(0) = 20\), \(s(1) = 20\alpha\) etc. However, taking \(s(0) = 1\) seems more natural, and gives the more satisfying bound of \(\# \sub^n(0) \leq \alpha^n\) below.

The result now follows. Indeed, the above showed that \(s(a) \geq 1\) for all \(a \in \N_0\); for \(a \in [0..7]\) these are listed, and the induction showed it also holds for all \(a \geq 8\). Since each letter has a score of at least \(1\), for any word \(w = a_1 \cdots a_n \in (\N_0)^\ast\) we have \(\# w = n = 1 + 1 + \cdots + 1 \leq s(a_1) + s(a_2) + \cdots + s(a_n) = s(w)\). In particular, \( \# \sub^n(a) \leq s(\sub^n(a)) = s(a) \alpha^n \).
\end{proof}

\subsection{Discussion, and variations on the example} \label{sec:counterexample discussion}
Let us give some context for the development of this counterexample. Initially, we observed and proved the same behaviour occurs for the substitution with \(a \mapsto 0 \ \lfloor a/2 \rfloor \ (a+1)\), for large enough \(a\), following some initial `stalling' terms \(0 \mapsto 1 \mapsto \cdots \mapsto m\). The idea here was to exploit the rapid decay of the term dividing letters by \(2\). A simple Haskell program was written that calculates lengths of superwords for a given substitution, where it was found to some initial surprise that the same behaviour appeared to persist with the much slower decaying \(a-2\) term in place of \(\lfloor a/2 \rfloor\). In this case, a proof similar to the above was found by bounding above using the more simply defined scores \(c_a = \beta^a\), when stalling to \(m = 7\), for a suitable constant \(\beta\). This method was too limited to extend down to the case of \(m=2\) though, whilst numerics suggested the example should still be non-uniquely ergodic. To extend down to the case of \(m=2\) treated in Theorem \ref{thm:non-UE example}, it was natural to try optimising the constants \(c_a\) recursively, which as we saw is sufficient (even though the bound \(\# w \leq s(w)\) is obviously far from optimal).

It seems important that there is an average decrease in letters between the `infinity columns' \(a-2\) and \(a+1\) in the substitution, with letters being pushed back to the region of slow growth more quickly than they advance. One might also consider the substitutions \(\sub\) defined by \(0 \mapsto 1 \mapsto \cdots \mapsto m\), and then \(a \mapsto 0 \ (a-1) \ (a+1)\) for \(a \geq m\). Even for relatively small \(m\), such substitutions take an extraordinarily long time to `catch up': taking \(m=4\), for instance, even at level \(1000\)-superwords, \(\sqrt[1000]{\# \sub^{1000}(0)} \approx 1.99439 < 2\). Despite this, as pointed out to us by Lorenzo Sadun, they should always eventually catch up to \(\# \sub^n(a) > 2^n\) for all \(a \in \A\), once \(n\) is large enough (relative to \(m\)). For instance, \(\sub^n(0) > 2^n\) holds once \(n \geq 8\), \(34\), \(140\) and \(601\) for \(m = 1\), \(2\), \(3\) and \(4\), respectively. This will make the resulting spectral radii \(r > 2\), and the substitution operators quasi-compact, so the subshifts are uniquely ergodic and NLFs exist (see results in the next section). Although they will be uniquely ergodic, convergence to ergodic averages is remarkably slow, see Remark \ref{rem:slow convergence}.

Further details of the above uniquely ergodic examples will be provided in future joint work with Sadun, as well as a closer analysis of the non-uniquely ergodic example, and variants, of Theorem \ref{thm:non-UE example}, which we now discuss just briefly (since our main interest here is simply that such non-uniquely ergodic can exist). As may be expected from how we proved Theorem \ref{thm:non-UE example}, there are two invariant ergodic measures, roughly speaking one defined by `small' labels and one by `large' labels. To be more explicit, given \(n \in \N\) and \(a\), \(b \in \A\), consider the quantity
\[
\mathrm{freq}^n(a,b) \coloneqq \frac{\text{number of occurrences of } a \text{ in } \sub^n(b)}{\# \sub^n(b)} \in [0,1].
\]
Let \(b \in \N_0\) be fixed and finite. Then, regardless of the exact choice of \(b\), for each \(a \in \A\) the above will converge, as \(n \to \infty\), to a value that geometrically decays in \(a\). For instance, \(\freq^n(0,b) \to 1 - \left(\sqrt[3]{\frac{1}{2} + \sqrt{\frac{31}{108}}} + \sqrt[3]{\frac{1}{2} - \sqrt{\frac{31}{108}}}\right) \approx 0.31767\) as \(n \to \infty\), for fixed \(b \in \N_0\), and of course \(\freq^n(\infty,b) = 0\) for all \(n \in \N\). In contrast, for \(b = \infty\), we have \(\freq^n(0,\infty) \to \frac{\sqrt{5}-1}{4} \approx 0.30902\) and \(\freq^n(\infty,\infty) \to \frac{3-\sqrt{5}}{4} \approx 0.19098\) as \(n \to \infty\).

We emphasise that this is not some artifact of our choice to compactify to include \(\infty\) in our alphabet. For \(a \in \N_0\), \(n \in \N\) and \(b \in \N_0\) sufficiently large (relative to both) we have \(\freq^n(a,b) = \freq^n(a,\infty)\), so fluctuation of these frequencies over small- and large-labelled \(n\)-superwords occurs in all elements of \(X\), even those without any \(\infty\) letters.

\begin{remark}
Mean ergodicity of \(T\) implies the existence of a NLF, which is generally weaker than strong power convergence (which in turn is weaker than quasi-compactness, in the primitive case), see Theorem \ref{thm:equivalence of operator conditions} and examples in \cite[Section 6]{MRW25}. Hence, Theorem \ref{thm:non-UE example} answers both Questions 4.9 and 6.14 in \cite{MRW25} in the negative: primitivity and existence of isolated points in the alphabet (or even all but one point being isolated!) are insufficient for the existence of a NLF, and certainly therefore also for any of the operator properties of mean ergodicity, strong power convergence or quasi-compactness.
\end{remark}

\begin{remark}\label{rem:slow convergence}
Consider the substitution \(\sub\) on \(\A = \N_\infty\), defined by \(0 \mapsto 1 \mapsto \cdots \mapsto 4\) and \(n \mapsto 0 (n-1) (n+1)\) for \(n \geq 4\). Since we compute that \(\# \sub^{601}(a) > 2^{601}\) for all \(a \in \A\) it follows (see the next section) that \(r > 2\), \(M\) is quasi-compact and \(X\) is uniquely ergodic. Despite this, ergodic averages that just count individual letters converge \emph{extremely} slowly. Considering level \(200\)-superwords,
\[
\mathrm{freq}^{200}(\infty,0) = 0 , \ \ \mathrm{freq}^{200}(\infty,\infty) \approx 0.1076 .
\]
By quasi-compactness, not only does one have uniform convergence of time-averages (in the sense of Theorem \ref{thm:UE characterisation}), there is even a sense in which it is uniform in the continuous function \(f \colon X \to E\), when of a particular form, which follows from uniform power convergence of \(T^n\) (see Theorem \ref{thm:equivalence of operator conditions}): consider all functions of the form \(f(w) = g(w_0)\), where \(g \in E\) (i.e., \(g \colon \A \to \R\) is continuous) with \(\|g\| \leq 1\). Then, for all \(\epsilon > 0\) there exists some \(k = k(\epsilon)\) (in particular, \emph{not} depending on \(f\)) and \(c(f) \in \R\) for which, for all \(n \geq k\) and \(w \in X\), we have \(\avg^n_w(f) = c(f) \pm \epsilon\). This follows from uniform power convergence of \(T^n\) (see Theorem \ref{thm:equivalence of operator conditions} and the paragraph below Proposition \ref{prop:QC=>UE}). For instance, take \(g \in E\) as the indicator function of \([q,\infty]\), for which it is clear that \(c(f) \to 0\) as \(q \to \infty\). But the frequency above shows that, for \(q\) large and averaging over \(200\)-superwords of label \(\infty\) (or finite label \(a > 200+q\)), which have length \(\approx 1.5 \times 10^{61}\), we do not even have \(\avg_w^{1.5 \times 10^{61}}(f)\) is within \(10\%\) of the correct average! This is rather surprising, given the relatively simple definition of \(\sub\).
\end{remark}

\section{Substitution reduction and the essential spectral radius}\label{sec:sub reduction}

In this section we will show how bounds, and often an exact formulae, can be found for the essential spectral radius of the operator \(M\), especially in the case that the alphabet consists of only finitely many accumulation points. This can sometimes be used to establish quasi-compactness of \(M\), from which one may deduce unique ergodicity of \(X\) when \(\sub\) is primitive.

\subsection{Substitution reduction and compact perturbations}
For the following operator theoretic definitions and properties, see for instance \cite{Arv02,Bro61}. Recall that an operator \(V \colon E \to E\) is called \textbf{compact} if \(V(B)\) is relatively compact, for \(B\) the closed unit ball in \(E\). Let
\[
\|M\|_\mathrm{ess} \coloneqq \inf\{ \|M - V\| \mid V \text{ is a compact operator} \} ,
\]
which is the operator norm of \(M\) when projected to the Calkin algebra (the quotient of the Banach algebra of bounded linear operators on \(E\) by the closed two-sided ideal of compact operators). The \textbf{essential spectral radius} of \(T\) may be defined by a Gelfand type limit:
\begin{equation}\label{eq:essential Gelfand}
r(M)_\mathrm{ess} \coloneqq \lim_{n \to \infty} \sqrt[n]{ \|M^n\|_\mathrm{ess} } = \inf_{n \in \N} \sqrt[n]{ \|M^n\|_\mathrm{ess} } .
\end{equation}
This number may also be understood as the spectral radius of \(M\) when projected to the Calkin algebra, or the common radius of the various definitions of the essential spectrum of \(M\). For example, it is also given by the smallest value for which every element of the spectrum of \(M\) with strictly larger modulus is an isolated eigenvalue of \(M\) with finite algebraic multiplicity. It is invariant under compact perturbation, that is, \(r_\mathrm{ess}(M) = r_\mathrm{ess}(M+V)\) for all compact \(V\), and for any \(\alpha \in \R\) we have \(r_\mathrm{ess}(\alpha M) = |\alpha| r_\mathrm{ess}(M)\).

The operator \(M\) is called \textbf{quasi-compact} if \(r_\mathrm{ess}(M) < r(M)\). Recalling that \(T = (1/r)M\), we have that \(M\) is quasi-compact if and only if \(T\) is, that is, \(r_\mathrm{ess}(T) < 1\). Note that quasi-compactness of \(T\) does not imply strong power convergence generally. Indeed, for the (irreducible) substitution \(a \mapsto bb\), \(b \mapsto aa\) on \(\A = \{a,b\}\), obviously \(M\) is quasi-compact (any finite-dimensional operator is compact), but for the function \(f \in E\) given by \(f(a) = 1\) and \(f(b) = -1\), we have \(T^n f = (-1)^n f\) is not convergent. However, for \emph{primitive} substitutions, quasi-compactness is extremely useful: the following (with a bit more) was stated in \cite[Proposition 6.6]{MRW25} although is essentially proved by an operator theoretic result, such as \cite[Prop.\ 5]{Abd75}.

\begin{proposition}\label{prop:QC=>UE}
If \(\sub\) is primitive and \(M\) (equivalently \(T\)) is quasi-compact then \(T\) is uniformly (thus also strongly) power convergent. In particular, \(X\) is uniquely ergodic.
\end{proposition}

By uniform power convergence, we mean that \(T^n\) converges with respect to the operator norm in the Banach algebra of bounded operators on \(E\). That is, there is a bounded operator \(P \colon E \to E\) (necessarily a projection to the subspace fixed by \(T\)) with \(\|T^n - P\| \to 0\) as \(n \to \infty\). Strong power convergence only requires that, for all \(f \in E\), we have \(\|T^n(f) - P(f)\| \to 0\) as \(n \to \infty\) (instead, uniform power convergence needs this convergence to be uniform over \(f\)).

The essential spectral radius measures the asymptotic growth of the `essentially infinite-dimensional part' of the operator \(M\) and, as such, one may suspect it is unaffected by `finitary adjustments' of the substitution. The following definition and result makes this explicit.

\begin{definition}\label{def:reduced substitution}
A substitution \(\tau\) is called a \textbf{reduction} of \(\sub\) when \(\tau\) may be defined by one of the following procedures:
\begin{itemize}
	\item[(\(\varepsilon_\mathrm{in}\))] erasing isolated inputs: given a finite set \(P \subseteq \A\) of isolated points, we define \(\tau(a) = \varepsilon\) if \(a \in P\), and \(\tau(a) = \sub(a)\) otherwise;
	\item[(\(\varepsilon_\mathrm{out}\))] erasing isolated outputs: given a finite set \(P \subseteq \A\) of isolated points, we define \(\tau(a)\) by deleting all occurrences of letters in \(P\) from \(\sub(a)\);
	\item[(\(\varepsilon_\mathrm{con}\))] erasing constants: if \(U \subseteq A\) is clopen, \(b \in \A\) and, for all \(u \in U\), we have \(b \triangleleft \sub(u)\), then for all \(a \notin U\) we define \(\tau(a) = \sub(a)\), and for all \(u \in U\) we define \(\tau(u)\) by deleting one occurrence of \(b\) from \(\sub(u)\).
\end{itemize}
Generally, call \(\tau\) a \textbf{reduction} of \(\sub\) if it can be derived by iteratively applying any finite number of the above procedures.
\end{definition}

\begin{remark}
Although (\(\varepsilon_\mathrm{in}\)) and (\(\varepsilon_\mathrm{out}\)) reductions clearly produce well-defined, continuous substitutions, there is a minor abuse in the last one. For instance, consider the (continuous) substitution \(\sub\) on \(\A = \N_\infty\) defined, for \(n\) even, by \( \sub(n) = (n+1) \ 0 \ \infty\) and, for \(n\) odd, by \(\sub(n) = \infty \ 0 \ (n+1)\). There is one occurrence of \(\infty\) in \(\sub(u)\) for all \(u \in U \coloneqq \A\). However, just deleting each would give, for \(n\) even, \(\tau(n) = (n+1) \ 0\) and, for \(n\) odd, \(\tau(n) = 0 \ (n+1)\), which is discontinuous. However, the ordering of letters is irrelevant to the substitution operators (which will be the focus of our results in this section). So one could rearrange the letters to derive a continuous substitution \(\tau\), without affecting the operators. In fact, it is preferable within this section to just consider these substitutions as unordered, that is as `finite multiset substitutions' (recall the discussion in Section \ref{sec:sub operators}), which bypasses this issue entirely. Thus, we should technically consider the reduction of a substitution in Definition \ref{def:reduced substitution} as merely a finite (unordered) multiset substitution, or always just talk in terms of the operators. However, we will overlook this minor technicality in what follows, as it does not meaningfully affect the arguments and should cause no confusion.
\end{remark}

The utility of the above definition is the following:

\begin{lemma}\label{lem:essential spectral radius unaffected by reduction}
If \(\tau\) is a reduction of \(\sub\), with substitution operators \(M_\tau\) and \(M\), respectively, then \(r_\mathrm{ess}(M_\tau) = r_\mathrm{ess}(M)\).
\end{lemma}

\begin{proof}
We show that \(M_\tau\) is a compact perturbation of \(M\), upon applying any of the three reduction procedures. First, let \(\tau\) be defined by an (\(\varepsilon_\mathrm{in}\)) reduction, for a given finite set \(P\) of isolated points. Let \(Q = \A \setminus P\) be the complementary points and define \(V_P \colon E \to E\) by 
\[
(V_P f)(a) =
\begin{cases}
f(a) & \text{for} \ a \in P \\
0    & \text{otherwise.}
\end{cases}
\]
Note that \(V_P \colon E \to E\), since \(\A = P \sqcup Q\) is a clopen partition. Clearly, \(V_P\) is linear and has image the \(\#P\)-dimensional subspace of functions supported on \(P\), so is finite rank and thus compact. Define \(V \coloneqq V_P \circ M\) which, as a composition of a bounded and compact operator, is also compact. With \(I\) the identity operator, \(I - V_P\) is defined by \((I - V_P)(f) = f |_Q\), where \((f |_Q)(a) = f(a)\) for \(a \in Q\) and \((f |_Q)(a) = 0\) otherwise. Thus,
\[
(M-V)(f) = ((I-V_P) \circ M)(f) = (Mf) |_Q .
\]
It follows that \(M - V = M_\tau\). Indeed, recall that \(\tau\) is defined identically to \(\sub\) on letters in \(Q\) and sends letters in \(P\) to the empty word which, from the operator's perspective, is equivalent to setting \((M_\tau(f))(p) = 0\) for \(p \in P\).

Now suppose that \(\tau\) is instead defined by (\(\varepsilon_\mathrm{out}\)) reduction. The proof is similar to the above: this time we define, instead, \(V \coloneqq M \circ V_P\). Then \((M-V)(f) = (M \circ (I-V_P))(f) = M (f|_Q)\). Thus, \(M-V = M_\tau\), since removing occurrences of letters in \(P\) from superwords has the same effect as setting \(f\) equal to \(0\) on such letters and then applying the substitution operator.

Finally, let \(\tau\) be defined by (\(\varepsilon_\mathrm{con}\)) reduction, `erasing a constant', with respect to a clopen set \(U\) on which each superword contains an occurrence of \(b \in \A\). Let \(\chi_U\) be the indicator function of \(U\), that is \(\chi_U(a) = 1\) for \(a \in U\) and \(\chi_U(a) = 0\) for \(a \notin U\). Since \(U\) is clopen, we have \(\chi_U \in E\). Then we define the operator \(V_U^b \colon E \to E\) as
\[
V_U^b(f) \coloneqq (f(b)) \cdot \chi_U .
\]
Clearly, \(V_U^b\) is linear and has rank one, since its image is spanned by \(\chi_U \in E\) (so, in particular, \(V_U^b\) is compact). Then \((V_U^b(f))(a) = 0\) if \(a \notin U\) and \((V_U^b(f))(a) = f(b)\) otherwise, thus
\[
((M - V_U^b)f)(a) =
\begin{cases}
(Mf)(a)        & \text{for } a \notin U \\
(Mf)(a) - f(b) & \text{for } a \in U . \\
\end{cases}
\]
It follows that \(M-V_U^b = M_\tau\), since the above has had the effect of removing one occurrence of \(b\) from the substitution of any letter in \(U\). Since applying any finite combination of the three reduction procedures defines a new operator that only differs by the addition of a compact operator, the essential spectral radius remains invariant.
\end{proof}

We now focus on substitutions over alphabets that have only finitely many accumulation points, where we will always be able to make drastic enough reductions to have effective comparisons with a finite substitution, whereupon calculations may be made. Denote the accumulation points by \(\A^{(\infty)} = \{\infty_1,\infty_2,\ldots,\infty_m\}\) and the isolated points by \(\A^\bullet = \A \setminus \A^{(\infty)}\).

\begin{definition}
A substitution \(\tau\) on an alphabet with finitely many accumulation points is said to be \textbf{\(\infty\)-reduced} if there are open neighbourhoods \(U_i\) of \(\infty_i\) satisfying the following:
\begin{enumerate}
	\item the \(U_i\) are mutually disjoint;
	\item defining \(P \coloneqq \A \setminus (U_1 \cup \cdots \cup U_m)\), for all \(p \in P\) we have \(\tau(p) = \varepsilon\);
	\item for all \(i \in [1..m]\) and \(b \in \A\), there is no open neighbourhood \(U_i' \subseteq U_i\) of \(\infty_i\) for which, for all \(u \in U_i'\), we have \(b \triangleleft \tau(u)\);
	\item for all \(a \in U_i\) we have \(\# \tau(a) \leq \# \tau(\infty_i)\) and, furthermore, for all \(i\) and \(j \in [1..m]\), and all \(u \in U_i\), the number of occurrences of letters of \(U_j\) in \(\tau(u)\) is at most the number of occurrences of \(\infty_j\) in \(\tau(\infty_i)\).
\end{enumerate}
In this case, \(\tau(\A^{(\infty)}) \subseteq (\A^{(\infty)})^\ast\), that is, accumulation points substitute to words consisting only of accumulation points. We define the substitution \(\overline{\tau}\) on the finite alphabet \(\A^{(\infty)}\) by restriction i.e., \(\overline{\tau}(\infty_i) \coloneqq \tau(\infty_i)\) for all \(\infty_i \in \A^{(\infty)}\).
\end{definition}

Condition (3) guarantees that there are no `constants' within any of the \(U_i\). Condition (4) may be thought of as requiring substitution of each \(U_i\) to essentially mimic substitution of the accumulation points, which it is already guaranteed to do sufficiently close to them, by continuity (although, further from the accumulation points, letters may get erased). One might also demand that all \(\tau(a)\) contain no letters in \(P\), by applying an (\(\varepsilon_\mathrm{out}\)) reduction, which would have the benefit of making the first requirement of Condition (4) unnecessary. However, the given definition allows for a slightly simpler presentation when applied to examples.

\begin{lemma}\label{lem:infinity-reduction exists}
Every substitution on an alphabet with finitely many accumulation points has a reduction to one that is \(\infty\)-reduced.
\end{lemma}

\begin{proof}
Since \(\A\) is Hausdorff, we may certainly choose mutually disjoint open neighbourhoods \(U_i\) of the \(\infty_i\). Since the complement of each \(U_i\) is then the union of the other open \(U_j\) and some (also open) isolated points, each \(U_i\) is also closed. Thus, by performing (\(\varepsilon_\mathrm{in}\)) reduction, with \(P\) the complement of \(U_1 \cup \cdots \cup U_m\), we may reduce to a substitution satisfying (1--2).

Now, supposing that (3) is not satisfied (say, for a given \(b \in \A\) and \(U_i'\)), we may first perform another (\(\varepsilon_\mathrm{in}\)) reduction to restrict the non-erasing letters to \(U_1 \cup \cdots \cup U_{i-1} \cup U_i' \cup U_{i+1} \cup \cdots \cup U_m\) (so that (1--2) are still satisfied), and then (\(\varepsilon_\mathrm{con}\)) reduction on \(U_i'\) to remove a single occurrence of \(b\) here. If (3) still fails, we may repeat this process again. Each time, the value of some \(\# \sub(\infty_i) \in \N_0\) is reduced by 1, so eventually (3) is satisfied.

Once (1--3) are satisfied, it is clear that \(\tau(\A^{(\infty)}) \subseteq (\A^{(\infty)})^\ast\). Indeed, if \(b \triangleleft \sub(\infty_i)\) for some \(b \in \A^\bullet\) then there exists some open neighbourhood \(U\) of \(\infty_i\) with \(b \triangleleft \sub(a)\) for all \(a \in U\), by continuity and the fact that \(\{b\}\) is open, contradicting (3). Also by continuity, for each \(\infty_i \in \A^{(\infty)}\), clearly there is a sufficiently small open neighbourhood \(U_i' \subseteq U_i\) of \(\infty_i\) so that, for all \(a \in U_i'\), the number of occurrences of \(U_j\)-letters in \(\tau(a)\) is equal to the number of occurrences of \(\infty_j\) in \(\tau(\infty_i)\). Erasing \(U_i\) down to \(U_i'\) of course does not affect properties (1--3), which can be achieved with (\(\varepsilon_\mathrm{in}\)) reductions again. We may repeat this for each accumulation point until (4) is satisfied.
\end{proof}

Generally, one may think of an \(\infty\)-reduced substitution to be one that only contains `columns with \(\infty\) limits' about each accumulation point (and with no such being constant).

\begin{example}
Let \(\A = \Z \cup \{-\infty,\infty\}\) be the two-point compactification of the integers. We define a substitution \(\sub \colon \A \to \A^+\) as follows:
\[
\sub(n) = 
\begin{cases}
0 \ 1                               & \text{for } n=0 \\
\min\{7,n\} \ (n+1) \ (-n) \ \infty & \text{for} \ n > 0 \\
0                                   & \text{for} \ n < 0 .
\end{cases}
\]
Of course, by continuity, we require \(\sub(-\infty) = 0\) and \(\sub(\infty) = 7 \infty (-\infty) \infty\). The first few iterations on \(0\) are:
\[
0 \mapsto 01 \mapsto 01 12(-1)\infty \mapsto 01 12(-1)\infty 12(-1)\infty 23(-2)\infty 0 7 \infty(-\infty) \infty \mapsto \cdots .
\]
This substitution is primitive, which may be easily checked. An example of an \(\infty\)-reduction is
\[
\tau(n) =
\begin{cases}
(n+1) \ (-n)  & \text{for } n \geq 8\\
\varepsilon & \text{otherwise} ,
\end{cases}
\]
with \(U_1 = [-\infty..-1]\) and \(U_2 = [8..\infty]\) in Definition \ref{def:reduced substitution}. It may be realised by an (\(\varepsilon_\mathrm{in}\)) reduction that erases superwords of letters in \(P = [0..7]\), then (\(\varepsilon_\mathrm{con}\)) erasing the constant \(7\), then \(\infty\) on \([8..\infty]\) and also (\(\varepsilon_\mathrm{con}\)) erasing the constant \(0\) on \([-\infty..-1]\).

Note that, although the original substitution was primitive, the induced finite substitution \(\overline{\tau}\) on \(\A^{(\infty)} = \{-\infty,\infty\}\) is erasing. It is given by \(\overline{\tau}(-\infty) = \varepsilon\) and \(\overline{\tau}(\infty) = \infty (-\infty)\) and has substitution operator represented by the matrix
\(\begin{pmatrix}
1 & 1 \\
0 & 0
\end{pmatrix}\), with eigenvalues \(1\) and \(0\), thus with spectral radius \(1\). As we will see below, since \(\tau\) maps isolated points to words of isolated points, \(r_\mathrm{ess}(M) = 1\) is the same value. As for any irreducible substitution, \(r(M) > 1\) (see Lemma \ref{lem:irreducible=>letter growth}), so \(M\) is quasi-compact and \(X\) is uniquely ergodic.
\end{example}

Generally, we may relate (and often identify) the essential spectral radius of an \(\infty\)-reduced substitution \(\tau\) with the spectral radius of the finite substitution \(\overline{\tau}\), by the following result:

\begin{theorem} \label{thm:spectral radius of reduction}
Let \(\tau\) be an \(\infty\)-reduced substitution, with substitution operator \(M\). Then, for all \(n \in \N\),
\begin{equation} \label{eq:essential norm equation}
\| M^n \|_\mathrm{ess} \leq \| M^n \| = \| M_\infty^n \| = \max_{a \in \A^{(\infty)}} \# \overline{\tau}^n(a) ,
\end{equation}
where \(M_\infty\) is the substitution operator of \(\overline{\tau}\). In particular, \(r_\mathrm{ess}(M) \leq r(M) = r(M_\infty)\).

Suppose, additionally, that \(\tau\) maps isolated points to isolated points, that is, \(\tau(\A^\bullet) \subset (\A^\bullet)^\ast\). Or, more generally, suppose that for all \(n \in \N\), all \(\infty_i \in \A^{(\infty)}\) and all open neighbourhoods \(U\) of \(\infty_i\), there exists some \(a \in U\) with \(\tau^n(a) \in (\A^\bullet)^\ast\). Then the inequality in Equation (\ref{eq:essential norm equation}) is an equality. Hence, in this case, we have
\[
r_\mathrm{ess}(M) = r(M_\infty) .
\]
\end{theorem}

\begin{proof}
Trivially, \(\| M^n \|_\mathrm{ess} \leq \| M^n \|\) and \( \|M^n\| = \max_{a \in \A} (\# \tau(a)) \geq \max_{a \in \A^{(\infty)}} (\# \tau(a)) = \|M_\infty^n\|\). Condition (4) of being \(\infty\)-reduced implies the reverse of the latter inequality. Indeed, by applying (4) inductively, it is clear that for each \(a \in U_i\) we have \(\# \tau^n(a) \leq \# \tau^n(\infty_i)\), since under each application of substitution, at best letters mimic substitution of the accumulation points, and at worst some letters are lost when substituting outside of the \(U_1 \cup \cdots \cup U_m\), whereupon further substitution erases them by Condition (2). This verifies Equation (\ref{eq:essential norm equation}). It then follows from Gelfand's Formula that \(r_\mathrm{ess}(M) \leq r(M) = r(M_\infty)\).

So let us additionally suppose that, for any power of the substitution, accumulation points are arbitrarily close to letters that substitute to words of only isolated letters. Given any \(n \in \N\), we wish to show that \(\| M^n \|_\mathrm{ess} = \| M^n \|\), in other words, compact perturbations of \(M^n\) cannot decrease the operator norm. So let \(V \colon E \to E\) be a compact operator. We let \(Y_1 \supset Y_2 \supset Y_3 \supset \cdots\) be a shrinking sequence of clopen sets \(Y_i \subseteq \A\) containing the accumulation points, with \(\A^{(\infty)} = \bigcap_{i = 1}^\infty Y_i\) (such a sequence exists, since \(\A\) is compact, Hausdorff and has only finitely many accumulation points, so it is countable). Consider the associated characteristic functions \(f_i \in E\), defined by \(f_i(a) = 1\) if \(a \in Y_i\) and \(f_i(a) = 0\) otherwise.

Let \(\epsilon > 0\). By compactness of \(V\), there is a subsequence of values \(k\) for which \(V(f_k) \to f\) in \(E\), for some \(f \in E\). Suppose that \(\|M^n\| = N \in \N_0\). As seen above, by Conditions (2) and (4) of being an \(\infty\)-reduction, this means that \(N = \# \tau^n(\infty_p)\) for any value of \(p\) maximising \(\# \tau^n(\infty_p)\). Thus, \(\# \tau^n(a) = N\) for all \(a \in U_p \cap Y_i\), at least for large enough \(i\), and no other letters substitute to longer words under \(\tau^n\) (over the whole alphabet).

Let \(k\) be in the above subsequence with \(\|V(f_k) - f\| \leq \epsilon\). Let \(a \in \A\) be sufficiently close to \(\infty_p\) that \(\# \tau^n(a) = N\), and each letter of \(\tau^n(a)\) is in \(Y_k\) and not an accumulation point, which exists by our additional running assumption. Since all letters of \(\tau^n(a)\) are isolated, we may choose an even larger \(k'\) so that all letters of \(\tau^n(a)\) are outside of \(Y_{k'}\) and, by remaining on the subsequence, this can be arranged so that still \(\|V(f_{k'}) - f\| \leq \epsilon\). Then \(g = f_k - f_{k'}\) is the continuous indicator function of the set \(Y_k \setminus Y_{k'} \subset \A^\bullet\), so \(g \in E\) with \(\|g\| \leq 1\). Now,
\[
(M^n(g) - V(g))(a) = [(M^n f_k)(a) - (M^n f_{k'})(a)] - (V f_k - V f_{k'})(a) .
\]
By assumption, each of \((V f_k)(a)\) and \((V f_{k'})(a) = f(a) \pm \epsilon\). Moreover, by construction \((M^n f_k)(a) = N = \| M^n \|\) and \((M^n f_{k'})(a) = 0\), which gives
\[
(M^n(g) - V(g))(a) = \|M^n\| - 0 - f(a) + f(a) \pm 2\epsilon = \|M^n\| \pm 2\epsilon .
\]
Thus, for all \(\epsilon > 0\) we have \(\|M^n - V\| \geq \|M^n\| - 2\epsilon\), so \(\|M^n - V\| \geq \|M^n\|\). Since \(V\) was an arbitrary compact operator, we have \( \|M^n\|_\mathrm{ess} = \|M^n\|\), as required. As this holds for all \(n \in \N\), we have \(r_\mathrm{ess}(M) = r(M)\) by the Gelfand type Formula (\ref{eq:essential Gelfand}) for \(r_\mathrm{ess}\).
\end{proof}

We summarise with the following, which we will demonstrate on some further examples.

\begin{corollary}\label{cor:essential spectral radius}
Let \(\sub\) be a substitution on an alphabet containing only \(m < \infty\) many accumulation points and let \(M\) denote its substitution operator. For all \(n \in \N\) the substitution \(\sub^n\) has some \(\infty\)-reduction \(\tau\), with associated finite substitution \(\overline{\tau}\) on \(\A^{(\infty)}\), whose substitution operator is denoted \(M_\infty\). Then \(r_\mathrm{ess}(M) \leq \sqrt[n]{r(M_\infty)}\), where \(r(M_\infty)\) is simply the leading eigenvalue of the \(m \times m\) (non-negative) matrix representing \(M_\infty\). Moreover, if \(\tau\) sends isolated points to words consisting of only isolated points (or the more general condition given in Theorem \ref{thm:spectral radius of reduction}), then \(r_\mathrm{ess}(M) = \sqrt[n]{r(M_\infty)}\).
\end{corollary}

\begin{proof}
The substitution \(\sub^n\), whose substitution operator is \(M^n\), has an \(\infty\)-reduction \(\tau\) by Lemma \ref{lem:infinity-reduction exists}. By Lemma \ref{lem:essential spectral radius unaffected by reduction}, the substitution operators of \(\sub^n\) and \(\tau\) have the same essential spectral radius. Then, by Theorem \ref{thm:spectral radius of reduction}, we obtain \(r_\mathrm{ess}(M^n) \leq r(M_\infty)\), and even \(r_\mathrm{ess}(M^n) = r(M_\infty)\) if \(\tau\) satisfies the additional condition defined in Theorem \ref{thm:spectral radius of reduction}. The result follows, since \(r_\mathrm{ess}(M^n) = (r_\mathrm{ess}(M))^n\).
\end{proof}

\begin{example}
This example shows why we incorporated the flexibility of a power in Corollary \ref{cor:essential spectral radius}. Let \(\A = \N_\infty\) and define the primitive substitution
\[
\sub =
\begin{cases}
0 \mapsto 0 \ 1             & \\
a \mapsto 0 \ (a+1) \ \infty & \text{for even } a \\
a \mapsto 0 \ (a+1) \ (a+1)  & \text{for odd }  a
\end{cases}
\]
where continuity dictates that \(\infty \mapsto 0 \infty \infty\). It has an \(\infty\)-reduction
\[
\tau =
\begin{cases}
0 \mapsto \varepsilon  & \\
a \mapsto (a+1) \ \infty & \text{for even } a \\
a \mapsto (a+1) \ (a+1)  & \text{for odd }  a .
\end{cases}
\]
The finite substitution \(\overline{\tau}\) is given by \(\infty \mapsto \infty \infty\) and so \(r_\mathrm{ess}(M) \leq 2\). We cannot deduce equality, though, since \(\tau\) does not map isolated letters to words of only isolated letters. Nor does it satisfy the weaker requirement of Theorem \ref{thm:spectral radius of reduction}. Indeed, \(\tau^2(a) = (a+2) (a+2) \infty \infty\) for \(a\) even and \(\tau^2(a) = (a+2) \infty (a+2) \infty\) for \(a\) odd, which in neither case consists of only isolated points for \(a \in [1..\infty)\). But this guides us to considering \(\sub^2\), which has an \(\infty\)-reduction \(0 \mapsto \varepsilon\) and \(a \mapsto (a+2) (a+2)\) for \(a > 0\). The induced finite substitution is still \(\infty \mapsto \infty \infty\), and since isolated points now substitute to words of isolated points we obtain \(r_\mathrm{ess}(M) = \sqrt{2}\) from Corollary \ref{cor:essential spectral radius}. Since \(\min_{a \in \A} \# \sub^2(a) = 5\), we have \(r(M) \geq \sqrt{5} > 2 = r_\mathrm{ess}(M)\), so \(M\) is quasi-compact and \(X\) is uniquely ergodic. 
\end{example}

\begin{example}
Let \(\sub\) be the counterexample to primitivity implying unique ergodicity on \(\A = \N_\infty\) in Section \ref{sec:non-UE}. An \(\infty\)-reduction \(\tau\) is given by \(\tau(n) = (n-2) \ (n+1)\) for all \(n \geq 2\). So the associated finite substitution \(\overline{\tau}\) on the one-element alphabet \(\{\infty\}\) is given by \(\infty \mapsto \infty \infty\), hence \(r_\mathrm{ess}(M) = 2\). This example is not quasi-compact (as \(X\) is not uniquely ergodic) so \(r(M) = r_\mathrm{ess}(M) = 2\).
\end{example}

\begin{example}
Let \(\A = \Z \cup \{-\infty, \infty\}\) be the two-point compactification of the integers and consider the substitution
\[
\sub =
\begin{cases}
0 \mapsto 0 \ 1           & \\
a \mapsto 0 \ (a+1) \ (-a) & \text{for } a > 0 \\
a \mapsto 0 \ (-a)       & \text{for } a < 0 . \\
\end{cases}
\]
Again, primitivity is trivial to check. An \(\infty\)-reduction is \(\tau(0) = \varepsilon\), \(\tau(a) = (a+1) (-a)\) for \(a > 0\) and \(\tau(a) = (-a)\) for \(a < 0\), which restricts to the finite substitution \(\overline{\tau}(\infty) = \infty (-\infty)\), \(\overline{\tau}(-\infty) = \infty\). This is the Fibonacci substitution, which has spectral radius the golden mean \(\varphi = (1+\sqrt{5})/2\). Since \(\tau\) maps isolated points to isolated points, \(r_\mathrm{ess}(M) = \varphi\) by Corollary \ref{cor:essential spectral radius}. Since \(\min_{a \in \A} \# \sub(a) = 2\), we have \(r(M) \geq 2 > \varphi = r_\mathrm{ess}\), so \(M\) is quasi-compact and \(X\) is uniquely ergodic.
\end{example}

\begin{example}
To prevent the result becoming too technical, we restricted the statement of Theorem \ref{thm:spectral radius of reduction} to substitutions on alphabets with finitely many accumulation points. However, the final argument showing \(\|M^n\| = \|M^n\|_\mathrm{ess}\) for sufficiently restricted substitutions applies more generally. For instance, consider the family of primitive substitutions from \cite{FGM24}, which were used to show that any inflation factor \(r > 2\) may be realised by a compact alphabet substitution. A substitution in this family is defined by a bounded sequence \((\gamma_i)_{i=0}^\infty\), with each \(\gamma_i \in \N_0\), \(\gamma_0 \neq 0\) and with bounded runs of \(0\)s, and is defined on \(\N_0\) by
\[
\sub = 
\begin{cases}
0 \mapsto \underbrace{0 \cdots 0}_{\gamma_0 \text{-many}} \ 1 ,            & \\
n \mapsto \underbrace{0 \cdots 0}_{\gamma_n \text{-many}} \ (n-1) \ (n+1) & \text{for } n > 0 . 
\end{cases}
\]
This substitution may be extended continuously to a compact alphabet \(\A\) in which \(\N_0\) is densely embedded as a set of isolated points although, depending on the sequence \((\gamma_i)_i\), there can be finitely, countably or uncountably many accumulation points. Lemma \ref{lem:essential spectral radius unaffected by reduction} still applies in any case, so we may freely (\(\varepsilon_\mathrm{out}\)) erase all \(0\) letters from superwords, since \(0 \in \A\) is isolated. So this reduction \(\tau\) of \(\sub\), given by \(\tau(0) = 1\), \(\tau(1) = 2\) and \(\tau(n) = (n-1) \ (n+1)\) for \(n \geq 2\), has substitution operator \(M_\tau\) satisfying \(r_\mathrm{ess}(M) = r_\mathrm{ess}(M_\tau)\). Since \(\tau\) restricts to a constant length 2 substitution on the accumulation points, but also maps isolated points to words of isolated points, it easily follows that \(r_\mathrm{ess}(M_\tau) = 2\), by applying an analogous argument to the end of the proof of Theorem \ref{thm:spectral radius of reduction}, answering a question in \cite[Section 7]{FGM24}. Since \(\# \sub(a) \geq 2\) for all \(a \in \A\), and eventually every \(n\)-superword contains a letter \(b\) with \(\# \sub(b) > 2\), clearly \(\min_{a \in \A} \# \sub^n(a) > 2^n\) for some \(n \in \N\), so \(r > 2\); in fact, an explicit formula for the the inflation factor \(r = \lambda > 2\) is given in \cite[Section 4]{FGM24}. In particular, \(r > r_\mathrm{ess}(M)\) so \(M\) is quasi-compact.
\end{example}

\subsection{Equivalence of operator properties and uniform superword growth for primitive substitutions with isolated points}\label{sec:equivalence of operator properties}

In \cite{MRW25}, varying degrees of regularity of the substitution operator \(T\) were shown to be possible for primitive substitutions on compact alphabets. For instance, sometimes \(T\) is quasi-compact, but there are also primitive substitutions where \(T\) is strongly power convergent but not quasi-compact. By the running example of Section \ref{sec:non-UE}, \(T\) may even fail the weaker property of mean ergodicity (see \cite[Definition 8.4]{EFHN15} or \cite[Section 5]{MRW25}). Let us now prove Theorem \ref{thm:equivalence of operator conditions} from the introduction, whose main conclusion is that either none or all of these properties hold for primitive substitutions with an isolated point, and that the question boils down to uniformity of word growth under substitution.

\begin{proof}[Proof of Theorem \ref{thm:equivalence of operator conditions}]
The implications down the list are all trivial or already proved elsewhere. Firstly, quasi-compactness of \(M\) (equivalently \(T\)) implies strong power convergence (for primitive substitutions), as stated in Proposition \ref{prop:QC=>UE}. Uniform implying strong power convergence, implying mean ergodicity of \(T\) are trivial. Operator theoretic results can be used to show that (4) implies (5), even without irreducibility, see \cite[Proposition 5.12]{MRW25}. It is easy to show that that (5) implies (6) using \(\# \sub^n(a) = (M^n \bbo)(a)\) and that there are constants \(c\), \(C > 0\) for which \(c \bbo \leq \ell \leq C\bbo\) by irreducibility of the substitution; see \cite[Lemma 4.29]{MRW25}. And clearly (6) implies (7), since for all \(n \in \N\) and \(a\), \(b \in \A\) we have \(\# \sub^n(a) / \# \sub^n(b) \geq \alpha / \beta\).

So, suppose additionally that \(\sub\) has a weak coincidence, that is, there exists some \(p \in \A\) and \(n \in \N\) for which, for all \(a \in \A\), we have \(p \triangleleft \sub^n(a)\). This always holds when \(\A\) contains an isolated point, since if \(p \in \A\) is isolated then \(\{p\}\) is open, so the required power exists by the definition of primitivity. Since \(M\) is quasi-compact if and only if \(M^n\) is, to simplify notation we may take \(n=1\) without loss of generality, by replacing \(\sub\) with its power \(\sub^n\). Since \(U = \A\) is clopen in \(\A\), we may define a reduction \(\tau\) of \(\sub\) by the (\(\varepsilon_\mathrm{con}\)) reduction that deletes one occurrence of \(p\) from every superword.

Using (7), let \(D > 0\) be such that, for all \(n \in \N\), \(a\) and \(b \in \A\), we have \(\# \sub^n(a) / \# \sub^n(b) \geq D\), let \(L \coloneqq \max_{a \in \A} \# \sub(a)\) and define \(q \coloneqq 1 - \frac{D}{L}\). We claim that, for all \(a \in \A\) and all \(n \in \N\), we have \(\# \tau^n(a) \leq q^n \# \sub^n(a)\). Trivially, this holds for \(n=0\), so suppose it holds for some \(n = k\). Let \(a \in \A\) be arbitrary and \(a_i \in \A\) spell the superword \(\tau(a) = a_1 a_2 \cdots a_n\). Then
\begin{align*}
\# \tau^{k+1}(a) = \# \tau^k(a_1) + \cdots + \# \tau^k(a_n) \leq q^k(\# \sub^k(a_1) + \cdots + \# \sub^k(a_n)) = \\
q^k( \# \sub^{k+1}(a) - \# \sub^k(p)) = q^k\left( \# \sub^{k+1}(a) - \# \sub^{k+1}(a) \left[\frac{\# \sub^k(p)}{\# \sub^{k+1}(a)} \right] \right) \leq \\
q^k\left( \# \sub^{k+1}(a) - \# \sub^{k+1}(a) \left[ \frac{D}{L} \right] \right) =  q^{k+1}\# \sub^{k+1}(a) ,
\end{align*}
as required for induction. The first inequality is by the induction assumption, the next equality comes from \(\sub^{k+1}(a)\) containing each of \(\sub^k(a_i)\) along with one extra string \(\sub^k(p)\), and the final equality is derived from the square-bracketed terms, as follows,
\[
\frac{\# \sub^k(p)}{\# \sub^{k+1}(a)} = \frac{\# \sub^k(p)}{\# \sub^k(a_1) + \cdots + \# \sub^k(a_n) + \# \sub^k(p)} \geq \frac{\min_{a \in \A} \# \sub^k(a)}{L \max_{b \in \A} \# \sub^k(b)} \geq \frac{D}{L} ,
\]
since \(n+1 = \# \sub(a) \leq L\). Thus, for all \(a \in \A\) and \(n \in \N\), we have \(\# \tau^n(a) \leq q^n \# \sub^n(a)\). Let \(M_\tau\) be the substitution operator of \(\tau\). Given \(n \in \N\), let \(a \in \A\) maximise \(\# \tau^n(a)\). Then
\[
\|M_\tau^n\| = \# \tau^n(a) \leq q^n \# \sub^n(a) \leq q^n \max_{b \in \A} \# \sub^n(b) = q^n\|M^n\| .
\]
Applying Gelfand's Formula,
\[
r(M_\tau) = \lim_{n \to \infty} \sqrt[n]{\|M_\tau\|^n} \leq \lim_{n \to \infty} q \sqrt[n]{\|M\|^n} = q \cdot r(M) < r(M) ,
\]
as \(q = 1 - \frac{D}{L} < 1\). Since (using Lemma \ref{lem:essential spectral radius unaffected by reduction}) \(r_\mathrm{ess}(M) = r_\mathrm{ess}(M_\tau) \leq r(M_\tau) < r(M)\), we have that \(M\) and thus \(T\) is quasi-compact, as required.
\end{proof}

\appendix
\section{Strong power convergence implies unique ergodicity}

Here, we provide an alternative, more direct proof to that in \cite{MRW25}, that strong power convergence of \(T\) implies unique ergodicity of \(X\) for irreducible substitutions, staying within the symbolic setting rather than using a NLF to pass to a geometric hull with \(\R\)-action.

As noted in the previous proof (see there for references), mean ergodicity of \(T\) is sufficient for the existence of a natural length function. However, for completeness of this section, we give the trivial proof of the following:

\begin{lemma}\label{lem:SPC=>NLF}
Suppose that \(T\) is strongly power convergent. Then \(\sub\) admits a natural length function.
\end{lemma}

\begin{proof}
Let \(\ell = \lim_n T^n(\bbo) \in E\). Then \(T\ell = \ell\) is a fixed point and, since \(T\) is a positive operator, clearly \(\ell \in K\). Moreover, since \(\|T^n(\bbo)\| = \|T^n\| \geq 1\) (\(1 \leq \sqrt[n]{\|T^n\|}\), limiting from above by Gelfand), we must have that \(\|\ell\| \geq 1\) and thus \(\ell \neq 0\).
\end{proof}

When \(T\) is strongly power convergent, the function \(P(f) \coloneqq \lim_{n \to \infty} T^n(f)\) is a bounded operator, by the Uniform Boundedness Principle, and a projection to the subspace of fixed points of \(T\), satisfying \(P^2 = P\) and \(P = PT\). We say that \(\mu \in K'\) is an \textbf{eigenmeasure} if \(T' \mu = \mu\) and \(\mu(\bbo) = 1\), a slightly different choice to \cite[Definition 4.24]{MRW25}, but the appropriate one here.

\begin{lemma}
Suppose that \(T\) is strongly power convergent and \(\sub\) is irreducible. Then \(T\) has a unique eigenmeasure.
\end{lemma}

\begin{proof}
Let \(\ell\) be a NLF, which exists by Lemma \ref{lem:SPC=>NLF} and may be taken as \(\ell = P(\bbo)\). Moreover, modulo rescaling, it is the unique fixed point of \(T\) by Theorem \ref{lem:NLF unique and positive}, so \(P\) is a projection to the one-dimensional subspace spanned by \(\ell\). Suppose an eigenmeasure \(\mu\) exists. Then for arbitrary \(f \in E\), with \(Pf = c\ell\), since \(\mu(f) = \mu(Tf) = \mu(T^2f) = \cdots\), by invariance and continuity we have \(\mu(f) = \mu(Pf) = \mu(c \ell) = c \mu(\ell) = c \mu(P \bbo) = c \mu(\bbo) = c\), the latter from the normalisation \(\mu(\bbo) = 1\). So there is at most one eigenmeasure. And there is one, since \(\mu\) may be simply defined by \(\mu(f) \coloneqq c\) here. Continuity \(\mu \in E'\) is clear, by boundedness of \(P\), and invariance follows from \(T'\mu(f) = \mu(Tf) = \mu(PTf) = \mu(Pf) = \mu(f)\). Positivity \(P(K) \subseteq K\) follows from that of \(T\) and, since \(\ell \geq 0\), we have \(\mu(f) \geq 0\) for all \(f \in K\), so \(\mu \in K'\) and is an eigenmeasure.
\end{proof}

\begin{lemma}\label{lem:SPC=>growth bound}
Suppose that \(\sub\) is irreducible and \(T\) is strongly power convergent. Take the natural length function \(\ell \coloneqq P(\bbo)\). Then, for all \(\delta > 0\) there exists some \(N = N(\delta)\) so that, for all \(n \geq N\) and all \(a \in \A\), we have
\[
\frac{\# \sub^n(a)}{r^n \ell(a)} = 1 \pm \delta .
\]
\end{lemma}

\begin{proof}
By irreducibility, \(\ell \in K_{>0}\), hence we may take \(c > 0\) so that, for all \(a \in \A\), we have \(\ell(a) > c\). Using strong power convergence of \(T\), given \(\delta > 0\), take \(N\) so that, for all \(n \geq N\), we have \(\|T^n(\bbo) - \ell\| \leq c \delta\). Then, for all \(a \in \A\), we have
\[
\frac{\# \sub^n(a)}{r^n} = \frac{(M^n \bbo)(a)}{r^n} = \ell(a) \pm c \delta ,
\]
from which the result follows by multiplying both sides by \(1/\ell(a) \leq 1/c\).
\end{proof}

\begin{theorem}
If \(\sub\) is irreducible and \(T\) is strongly power convergent, then \(\sub\) admits a natural length function and \(X\) is uniquely ergodic.
\end{theorem}

\begin{proof}
Existence of the NLF has already been established (even without irreducibility) and may be taken as \(\ell = P(\bbo)\). Let \(f \colon X \to \R\) be continuous and \(\epsilon > 0\) be arbitrary. We aim to show uniform convergence of the time-averages \(\avg^n_w(f)\). Of course, this holds for \(f\) if and only if it holds for any other rescaling \(Cf\), for \(C \neq 0\), so without loss of generality \(\|f\| \leq 1\).

Let \(\delta > 0\); later, we will show how \(\delta = \delta(\epsilon)\) may be defined explicitly in terms of \(\epsilon\) at this point. Let \(N = N(\delta)\) be as in Lemma \ref{lem:SPC=>growth bound}; by taking \(N\) larger still, we may assume the reciprocal there also satisfies \(r^n \ell(a) / \# \sub^n(a) = 1 \pm \delta\) for all \(n \geq N\). For some sufficiently large \(k \geq N\), there is a sufficiently fine open cover \(\mathscr{C} = \{U_i\}\) of \(\A\) satisfying the following: if \(a\) and \(b\) belong to a common \(U_i\) then
\begin{enumerate}
	\item \(\# \sub^k(a) = \# \sub^k(b)\);
	\item for all \(w\), \(w' \in X\) and intervals \(I\), \(I' \subseteq \Z\) with \(w_I = \sub^k(a)\) and \(w'_{I'} = \sub^k(b)\) (necessarily, by the above, with \(\#I = \#I'\)), then \( \avg_w^I(f) = \avg_{w'}^{I'}(f) \pm \delta\).
\end{enumerate}
Indeed, because all \(k\)-superwords grow in length, by taking a large \(k\) we may ensure that all but an arbitrarily small proportion of the letters in any \(k\)-superword lie far from the boundary. Then, by ensuring that \(a\) and \(b\) are sufficiently close (taking a sufficiently fine cover, relative to \(k\)), we have \(\# \sub^k(a) = \# \sub^k(b)\) and that all constituent letters of these superwords are pairwise as close as desired. Then, by continuity of \(f\), its values on the vast majority of letters (far from the boundary) are also pairwise close, so the corresponding time-averages are close. By compactness, we may take the open cover \(\mathscr{C} = \{U_i\}\) to be finite.

For all \(a \in \A\), choose some \(w_a \in X\) with \((w_a)_{[0,m)} = \sub^k(a)\), where \(m \coloneqq \# \sub^k(a)\). Such \(w_a\) exists by Corollary \ref{cor:language contains superwords}. Define \(g \colon \A \to \R\) by \(g(a) \coloneqq \avg_{w_a}^m(f)\). Unfortunately, this need not be continuous. However, it is \(\delta\)-close to a continuous function, constructed as follows: choose a partition of unity \(\{\psi_i \colon \A \to [0,1]\}\) subordinate to \(\mathscr{C}\) (where \(\psi_i\) is supported on \(U_i\)) and let \(\widetilde{g} \colon \A \to \R\) be defined by
\[
\widetilde{g}(a) \coloneqq \sum_{U_i \in \mathscr{C}} \psi_i(a) g_i  , \ \text{ where } \ g_i \coloneqq \inf_{x \in U_i} g(x) .
\]
Each \(g_i \in [-1,1]\) here, since the averages \(g(x)\) are of a function with \(\|f\| \leq 1\), which means that \(\|\widetilde{g}\| \leq 1\); note that defining the \(g_i\) as infina here is not important, we may take each as any number between the infimum and supremum of \(g(U_i)\). As a finite linear combination of continuous functions, \(\widetilde{g}\) is continuous. Moreover, for all \(a \in \A\), we have \(\widetilde{g}(a)= g(a) \pm \delta\). Indeed, let \(J = J_a\) be the set of indices for which \(a \in U_i\). Then, by our demands of \(\mathscr{C}\) above, \(|g_i - g(a)| \leq \delta\) for all \(i \in J\), and \(\psi_i(a) = 0\) for \(i \notin J\), so
\[
|\widetilde{g}(a) - g(a)| = \left| \left(\sum_{j \in J} \psi_j(a) g_j \right) - g(a) \right| = \left| \left(\sum_{j \in J} \psi_j(a) (g_j - g(a)) \right) \right| \leq \sum_{j \in J} \psi_j(a) \delta = \delta .
\]
Now define \(h(a) \coloneqq \ell(a) \cdot \widetilde{g}(a)\), which is a continuous function \(h \colon \A \to \R\).

By strong power convergence, we may apply \(T\) to \(h \in E\) sufficiently many times, say \(p\) with \(p \geq N\), so that \(\|T^p(h) - P(h)\| \leq \delta\), where \(P(h) = c\ell\) for some \(c \in \R\). For all \(a \in \A\), we have \((T^p h)(a) = c\ell(a) \pm \delta\) and thus
\begin{equation}\label{eq:power close to constant}
\frac{(T^p h)(a)}{\ell(a)} = c \pm c_1 \delta
\end{equation}
where \(c_1 \coloneqq \sup_{a \in \A} (1/\ell(a))\). Recall that we also chose \(N\) in such a way that, for all \(a \in \A\) and \(n \geq N\) (which includes \(n = p\) and \(n = k\)), we have
\begin{equation}\label{eq:superword length estimate}
\frac{\# \sub^n(a)}{ \ell(a) r^n} = 1 \pm \delta = \frac{ \ell(a) r^n}{\# \sub^n(a)} .
\end{equation}
Take any \((k+p)\)-superword, which we may write as a concatenation \(\sub^{k+p}(a) = \sub^k(a_1) \cdots \sub^k(a_m)\) of \(k\)-superwords (with \(m = \# \sub^p(a)\) and \(a_i \in \A\)). Then, by Equation (\ref{eq:power close to constant}) and definition of \(T\), we have
\[
\frac{1}{r^p \ell(a)} \sum_{j=1}^m h(a_i) = \frac{1}{r^p \ell(a)}(M^p h)(a) = \frac{(T^p h)(a)}{\ell(a)} = c \pm c_1\delta .
\]

Let us now consider averages of \(f\) over \((k+p)\)-superwords: suppose that \(w \in X\) and \(I\) is some interval with \(w_I = \sub^{k+p}(a)\), for some \(a \in \A\), where \(\# I = \# \sub^{k+p}(a) \eqqcolon L\). Let us denote \(\sub^p(a) = a_1 a_2 \cdots a_m\), with letters \(a_i \in \A\) and \(m \coloneqq \# \sub^p(a)\). We also denote \(L_i \coloneqq \# \sub^k(a_i)\).

We may split the average calculation (Equation (\ref{eq:split average})) into \(k\)-superwords as
\[
\avg_w^I = \avg_w^I(f) = \frac{L_1}{L} \avg_w^{I_1}(f) + \frac{L_2}{L} \avg_w^{I_2}(f) + \cdots + \frac{L_m}{L} \avg_w^{I_m}(f) 
\]
where \(I = I_1 \sqcup \cdots \sqcup I_m\) is the canonical splitting of the interval according to the positions of the \(k\)-superwords \(\sub^k(a_1)\), \(\sub^k(a_2)\), \ldots, \(\sub^k(a_m)\), so that each \(\#I_i = L_i\). By Property (2) stated at the start of this proof and our definition of \(g\) below it, each \(\avg_w^{I_i}(f) = g(b_i) \pm \delta\) and thus
\[
\avg_w^I = \left(\sum_{i=1}^m \frac{L_i}{L} g(a_i)\right) \pm \delta .
\]
Similarly, since \(g = \widetilde{g} \pm \delta\), it follows that
\[
\avg_w^I = \sum_{i=1}^m \left( \frac{L_i}{L} \widetilde{g}(a_i) \right) \pm  2\delta = \frac{r^k}{L} \sum_{i=1}^m \left( \frac{L_i}{r^k} \widetilde{g}(a_i) \right) \pm  2\delta.
\]
Now, by the definition of \(L_i\) and Equation (\ref{eq:superword length estimate}),
\[
\frac{L_i}{r^k} = \frac{\# \sub^k(a_i)}{r^k} = \ell(a_i) \pm c_2 \delta 
\]
where \(c_2 \coloneqq \sup_{a \in \A} \ell(a)\). Thus, since \(\| \widetilde{g} \| \leq 1\), we have
\begin{equation} \label{eq:avg estimate}
\avg_w^I = \frac{r^k}{L} \sum_{i=1}^m \ell(a_i) \widetilde{g}(a_i) \pm  \delta\left( 2 + c_2 m \frac{r^k}{L} \right) = \frac{r^k}{L} \sum_{i=1}^m h(a_i) \pm  \delta p_1
\end{equation}
where \(p_1 \coloneqq 2 + c_2 m r^k / L\). Regarding the latter term, using Equation (\ref{eq:superword length estimate}) again,
\[
m \frac{r^k}{L} = \# \sub^p(a) \frac{r^k}{ \# \sub^{k+p}(a)} = \left(\frac{\# \sub^p(a)}{r^p \ell(a)} \right) \left( \frac{r^{k+p} \ell(a)}{\# \sub^{k+p}(a) } \right) = 1 \pm (2\delta + \delta^2) 
\]
so the error \(\pm \delta p_1\) in the Equation (\ref{eq:avg estimate}) is at most \(\pm \delta p_2\) with \(p_2 \coloneqq 2 + c_2(2\delta + \delta^2)\). Similarly, multiplying Equation (\ref{eq:avg estimate}) (with error \(\pm \delta p_2\)) by
\[
1 = \frac{\# \sub^{k+p}(a)}{\ell(a) r^{k+p}} \pm \delta = \frac{L}{\ell(a) r^{k+p}} \pm \delta ,
\]
using that \(|\avg_w^I| \leq 1\) (from \(\|f\| \leq 1\)), we have
\[
\avg_w^I = \frac{1}{r^p \ell(a)} \sum_{i=1}^m h(b_i) \pm \delta (1 + p_2 + \delta p_2) = \frac{T^p h(a)}{\ell(a)} \pm \delta p_3,
\]
where \(p_3 \coloneqq 1 + p_2 + \delta p_2\). Here, we use the trivial fact that if \(x = y \pm \delta_1\), \(|x| \leq 1\) and \(1 = \alpha \pm \delta_2\), then \(x = \alpha y \pm (\delta_1 + \delta_2 + \delta_1 \delta_2)\); indeed, \(|x-\alpha y| = |\alpha(x-y) + (x-\alpha x)| \leq |\alpha| |x-y| + |1-\alpha||x| \leq (1+\delta_2)\delta_1 + \delta_2\) (which we also used above in estimating \(mr^k / L\)).

By Equation (\ref{eq:power close to constant}), we thus have
\[
\avg_w^I(f) = c \pm \delta p_4 ,
\]
where \(p_4 = p_3 + c_1\), so the above error is \(\delta p_4\) where \(p_4\) is a polynomial expression in \(\delta\) that only depends on the constants \(c_1\) and \(c_2\) (which themselves only depend on \(\sub\)). In particular, by taking \(\delta = \delta(\epsilon) > 0\) sufficiently small at the start of the argument, we may ensure that
\begin{equation} \label{eq:superword average bound}
\avg_w^I(f) = c \pm \frac{\epsilon}{2}
\end{equation}
for any interval \(I\) spanning a \((k+p)\)-superword of \(w\).

The required conclusion, for general sufficiently long intervals (not just \((k+p)\)-superwords) is now quite obvious. To spell out the details, take \(I\) to be any (soon, reasonably long) finite interval, and partition \(I = I_1 \sqcup I_2 \sqcup \cdots \sqcup I_s \sqcup R\) where each \(I_i\) is a set of indices covered by a \((k+p)\)-superword, with \(R\) the remaining indices. By `representability' (i.e., every \(w \in X\) is given by a shift of \(\sub^{k+p}(w')\) for some \(w' \in X\), see \cite[Proposition 3.21]{MRW25}), the set \(R\) may be taken as a union of at most two intervals, each of length at most \(\max_{a \in \A} \# \sub^{k+p}(a)\). So denoting \(l \coloneqq \# I\) and \(q \coloneqq \# R\), we have \(q \leq 2 \max_{a \in \A} \# \sub^{k+p}(a)\), a term that does not depend on \(l\), and
\[
\avg_w^I(f) = \frac{\# I_1}{l}\avg_w^{I_1}(f) + \cdots + \frac{\# I_s}{l}\avg_w^{I_s}(f) + \frac{q}{l}\avg_w^R(f) .
\]
From Equation  (\ref{eq:superword average bound}), every \(\avg_w^{I_i} = c \pm \epsilon/2\) and, since \(\|f\| \leq 1\), we have \(\avg_w^R \leq 1\) so that
\[
\avg_w^I(f) = \left( \frac{l-q}{l} \right) c \pm \left(\left(\frac{l-q}{l}\right) \frac{\epsilon}{2} + \frac{q}{l}\right) .
\]
As \(l \to \infty\) we have \((l-q)/q \to 1\) and \(q/l \to 0\) so that, for sufficiently large \(l\) and all intervals \(I\) with \(\#I = l\), we have \(\avg_w^I(f) = c \pm \epsilon\). Thus, time-averages converge uniformly to a limit, establishing unique ergodicity. Technically, in the above, the value \(c\) was allowed to depend on the value of \(\epsilon > 0\) considered. However, we have shown that for all \(\epsilon > 0\) there is some \(c = c_\epsilon \in \R\) and some \(l_\epsilon \in \N\) so that, for all \(w \in X\) and all intervals \(I\) of length at least \(l_\epsilon\), we have \(\avg_w^I \in [c_\epsilon - \epsilon , c_\epsilon + \epsilon]\). Without loss of generality (increasing any terms if needed), \((l_{(1/n)})_n\) is an increasing sequence, so it follows that the intervals \([c_{1/n} - 1/n , c_{1/n} + 1/n]\) are nested, hence that the time-averages converge uniformly to the unique common intersection point.
\end{proof}

\bibliography{biblio}
\bibliographystyle{alpha}

\end{document}